\documentclass[12pt,oneside]{article}
\usepackage[margin=1in,footskip=0.25in]{geometry}

\usepackage{amssymb}

\usepackage{amsmath}

\usepackage{amsfonts}

\usepackage{longtable}

\newtheorem{thm}{Theorem}[section]
\newtheorem{lem}[thm]{Lemma}

\newtheorem{prp}[thm]{Proposition}

\newtheorem{cor}[thm]{Corollary}

\newtheorem{dfn}[thm]{Definition}

\newtheorem{example}[thm]{Example}

\newtheorem{fact}[thm]{Fact}

\newenvironment{theoremno}
{\begin{trivlist} \item \textbf{Theorem:}~}{\end{trivlist}}

\newtheorem{clm}[thm]{Claim}

\newenvironment{proof}
{\begin{trivlist}  \item \textsc{Proof:}~} {\hfill $\Box$
\end{trivlist}}

\newenvironment{proof of claim}
{\begin{trivlist}  \item \textsc{Proof of Claim:}~} {\hfill $\Box$
\end{trivlist}}

\def \inter{\operatorname{int}}

\def \gr{\operatorname{Graph}}

\def \id{\operatorname{id}}

\def \cl {\operatorname{cl}}

\def \cR{\mathcal R}
\def \cX{\mathcal X}
\def \R{R}

\def \bR{\mathbb{R}}

\def\Ind#1#2{#1\setbox0=\hbox{$#1x$}\kern\wd0\hbox to 0pt{\hss$#1\mid$\hss}
\lower.9\ht0\hbox to 0pt{\hss$#1\smile$\hss}\kern\wd0}

\def\Notind#1#2{#1\setbox0=\hbox{$#1x$}\kern\wd0\hbox to 0pt{\mathchardef
\nn=12854\hss$#1\nn$\kern1.4\wd0\hss}\hbox to
0pt{\hss$#1\mid$\hss}\lower.9\ht0 \hbox to
0pt{\hss$#1\smile$\hss}\kern\wd0}

\def \typ{\operatorname{tp}}

\def \di{disc}

\newcommand{\diam}{Diam}

\def \cl {\mathrm{cl}} \def \gR{\mathbb{R}^{>}} \def \cG{\mathcal G} \def \std{\operatorname{st}}  \def \an{an} \def \ran{\overline{\mathbb{R}}_{\an}} \def \ranex{\overline{\mathbb{R}}_{\an,\exp}}

\begin{document}

\title{On the Topology of Metric Spaces Definable in o-minimal expansions of fields.}
\author{Erik Walsberg}
\maketitle

\begin{abstract}
We study the topology of metric spaces which are definable in o-minimal expansions of ordered fields.
We show that a definable metric space either contains an infinite definable discrete set or is definably homeomorphic to a definable set equipped with its euclidean topology.
This implies that a separable metric space which is definable in an o-minimal expansion of the real field is definably homeomorphic to a definable set equipped with its euclidean topology.
We show that almost every point in a definable metric space has a neighborhood which is definably homeomorphic to an open definable subset of euclidean space.
We also show that over an expansion of the real field, the completion of a definable metric space is definable and characterize the topological dimension of a definable metric space.
\end{abstract}

\section{Introduction}

Throughout this paper $\cR = (R,+, \times, \leqslant,\ldots)$ is an o-minimal expansion of a field.
 By ``definable'' we mean ``$\cR$-definable, possibly with parameters from $R$''.
 We let $e$ be the usual euclidean metric on $R^k$ and refer to $(R^k,e)$ as ``euclidean space''.
The \textbf{euclidean topology} on a definable $Z \subseteq R^k$ is the topology induced by $e$.
A \textbf{definable metric space} is a pair $(X,d)$ consisting of a definable set $X \subseteq R^k$ and a definable $R$-valued metric $d$ on $X$.
That is, $d$ is a definable function $X^2 \rightarrow R^>$ such that for all $x,y,z \in X$:
\begin{enumerate}
\item $d(x,y) = 0$ if and only if $x = y$.
\item $d(x,y) = d(y,x)$.
\item $d(x,z) \leqslant d(x,y) + d(y,z)$.
\end{enumerate}
A \textbf{definable psuedometric} is a pair $(X,d)$ which satisfies (2) and (3) but not necessarily $(1)$.
We associate a definable metric space $(X/\!\sim, d')$ to a definable psuedometric where $x \sim x'$ if and only if $d(x,x') = 0$.
As o-minimal expansions of fields admit elimination of imaginaries the quotient $X/\sim$ may be realized as a definable set.

The main results of this paper are the following basic dichotomy, Theorem~\ref{main} below:
\begin{theoremno}
Let $(X,d)$ be a definable metric space.
Exactly one of the following holds:
\begin{enumerate}
\item There is an infinite definable $A \subseteq X$ such that $(A,d)$ is discrete.
\item There is a definable $Z \subseteq R^k$ and a definable homeomorphism $(X,d) \rightarrow (Z,e)$.
\end{enumerate}
\end{theoremno}
The following genericity result, Theorem~\ref{cor:localhomeo2} below,
\begin{theoremno}
Let $(X,d)$ be a definable metric space.
There is a $d$-open dense subset $U \subseteq X$ such that every $p \in U$ has a $d$-open neighborhood $U$ such that $(U,d)$ is definably homeomorphic to an open subset of euclidean space.
\end{theoremno}
and the following characterization of topological dimension, Theorem~\ref{cor:lasttopdim} below.
\begin{theoremno}
Suppose that $\cR$ expands the real field and that $(X,d)$ is a definable metric space.
The topological dimension of $(X,d)$ is the maximal $l$ for which there is a definable continuous injection $([0,1]^l,e) \rightarrow (X,d)$.
\end{theoremno}
Note that the previous theorem implies that the topological dimension of a definable metric space is no greater then the o-minimal dimension of the underlying definable set.
The next two sections are devoted to examples of definable metric spaces.
We give two kinds of examples.
In Section~\ref{section:families} we give examples of definable metric spaces associated to definable families of sets and functions.
In Section~\ref{section:noeu} we give examples of definable metric spaces which are topologically unlike definable sets.

The present paper is a somewhat modified portion of the authors thesis.
The author thanks his advisor Matthias Aschenbrenner for support and for greatly improving the author's writing.
The mistakes and deficiencies which remain are of course the authors.

\section{Notation, Tame Expansions, and Technical Lemmas}
\subsection{Basic Definitions and Notation}
We denote the o-minimal dimension of a definable set $X$ by $\dim(X)$.
We say that a property holds for all $0 < t \ll 1$ if there is a $\delta > 0$ such that the property holds for all $0 < t < \delta$.
Given a function $f:A \to B$ we let $\gr(f) \subseteq A \times B$ be the graph of $f$.
Throughout, $R^>$ is the set of positive elements of $R$ and $R^\geqslant$ is the set of nonnegative elements of $R$.
Let $X$ be a definable set.
We say that a definable $A \subseteq X$ is \textbf{almost all} of $X$ if $\dim(X \setminus A) < \dim(X)$.
We say that a property holds \textbf{almost everywhere} on $X$ if the property holds on a definable subset of $X$ which is almost all of $X$.

Let $X$ be a topological space.
Given a set $A \subseteq X$ we let $\cl(A)$ be the closure of $A$, let $\inter(A)$ be the interior of $A$, and $\partial(A) := \cl(A) \setminus A$ be the frontier of $A$.
Let $(X,d)$ and $(X',d')$ be definable metric spaces.
Given a definable $A \subseteq X$ we let
$$ \diam_d(A) := \sup\{ d(x,y) : x,y \in A \}. $$
Let $f: X \to X'$ be a map.
We say that $f$ is
\textbf{uniformly continuous} if for some function $g: R^{>}\rightarrow R^{>}$ such that $g(t)\rightarrow 0$ as $t\rightarrow 0^+$ we have:
 $$d'(f(x),f(y))\leqslant g(d(x,y)) \quad\quad \text{for all } x,y \in X.$$  
The map $f$ is a \textbf{uniform equivalence} if it is surjective and it satisfies
$$g_1(d(x,y))\leqslant d'(f(x),f(y))\leqslant g_1(d(x,y)) \quad \text{for all } x,y \in X. $$
for two functions $g_1,g_2 : R^> \rightarrow R^>$ such that $g_1(t),g_2(t) \rightarrow 0$ as $t \rightarrow 0^+$.
Uniform equivalences are homeomorphisms.
We say that $f$ is \textbf{distance decreasing}
if
$$ d'(f(x),f(y) \leqslant d(x,y) \quad \text{for all } x,y \in X. $$
A \textbf{path} in $X$ is a definable function $\gamma: R^> \rightarrow X$.
Note that a path in $X$ need not be continuous with respect to the $d$-topology on $X$.
The path $\gamma$ in $X$ \textbf{converges} to $x \in X$ if for every $\epsilon>0$ there is a $\delta>0$ such that if
$0<t<\delta$ then $d(\gamma(t),x)<\epsilon$.
We also say that $\gamma$ has limit $x$ as $t \rightarrow 0^+$.
The limit of $\gamma$ is unique, provided that it exists.
A path in $X$ converges if it converges to some $x\in X$.
\subsection{Lemmas}
We give a few easy technical lemmas which will prove useful.
\setcounter{thm}{0}\begin{lem}\label{lem:c}
Let $X$ be a definable set and let $f: X \rightarrow R^{>}$ be a definable function.
There is an open $V \subseteq X$ and a $\delta > 0$ such that $f(x) > \delta$ for all $x \in V$.
\end{lem}

\begin{proof}
Let $p \in X$ be a point at which $f$ is continuous.
Let $V$ be a definable neighborhood of $p$ such that $f(x) > \frac{1}{2}f(p)$ for all $x \in V$.
\end{proof}
\begin{lem}\label{lem:e}
Let $X$ be a definable set.
Let $A \subseteq X \times R^{>}$ be a definable set such that $X \times \{0\} \subseteq \cl(A)$.
There is an open $V \subseteq X$ and a $\delta > 0$ such that $V \times (0, \delta) \subseteq A$.
\end{lem}

\begin{proof}
Let $g: X \rightarrow R$ be given by
$$ g(x) = \sup\{ t \in (0,1) : \{x\} \times (0,t) \subseteq A\} \quad \text{for all } x \in X. $$
If $U \subseteq X$ is open and $g(x) = 0$ for all $x \in U$ then $U \times \{0\}$ does not lie in the closure of $A$.
Thus there is an open $U \subseteq X$ such that $g(x) > 0$ for all $x \in U$.
The lemma follows by applying Lemma~\ref{lem:c}.
\end{proof}
Let $A, B \subseteq R^k$ be definable.
The \textbf{Hausdorff distance} $d_H$ between $A$  and $B$ is the infimum of all $\delta \in R^> \cup\{ \infty\}$ such that for every $a \in A$ there is a $b \in B$ satisfying $\| a - b \| \leqslant \delta$ and for all $b \in B$ there is a $a \in A$ such that $\| a - b \| \leqslant \delta$.
\begin{lem}\label{lem:dez}
Let $\{ A_t  : t \in R^> \}$ be a definable family of subsets of $[0,1]^k$, let $A$ be the set of $(x,t) \in [0,1]^k \times R^\geqslant$ such that $x \in A_t$ and let
$$ A_0 = \{ x \in [0,1]^k : (x,0) \in \cl(A)\}.$$
Then $A_t$ converges in the Hausdorff metric to $A_0$ as $t \rightarrow 0^+$. 
\end{lem}
\begin{proof}
We show that $d_H(A_t, A_0) \rightarrow 0$ as $t \rightarrow 0^+$.
Suppose otherwise.
Then there is a $\delta \in R^>$ such that $d_H(A_t, A_0) > \delta$ when $t$ is sufficiently small.
Then one of the following must hold:
\begin{enumerate}
\item[(1)] If $0 < t \ll 1$ then there is a $x \in A_0$ such that $\| x - y\| > \delta$ for all $y \in A_t$.
\item[(2)] If $0 < t \ll 1$ then there is a $y \in A_t$ such that $\| y - x \| > \delta$ for all $x \in A_0$.
\end{enumerate}
Suppose that $(1)$ holds.
After replacing $\delta$ with a smaller element of $R^>$ if necessary and applying definable choice we let $\gamma$ be a path in $A_0$ such that:
$$ \| \gamma(t) - y \| > \delta \quad \text{for all } 0 < t < \delta, y \in A_t.$$
As $A_0$ is a closed subset of $[0,1]^k$, $\gamma$ must have a limit $\gamma_0$ in $A_0$ as $t \rightarrow 0^+$.
If $\| \gamma(t) - \gamma_0 \| < \frac{\delta}{2}$ then the triangle inequality implies that if $0 < t < \delta$ then $\| \gamma_0 - y \| > \frac{\delta}{2}$ holds for all $y \in A_t$.
As $\gamma_0$ lies in the closure of $A$ there is a $y \in A$ such that $\| y - \gamma_0\| < \frac{\delta}{2}$.
Then $y \in A_t$ for $0 < t < \frac{\delta}{2}$.
This gives a contradiction.
A similar argument produces a contradiction from $(2)$.
\end{proof}

\subsection{Tame Expansions}
We make use of the theory of tame pairs of o-minimal structures.
We recall the necessary ingredients here.
We assume that the language of $\cR$ contains a symbol for every definable function $R \rightarrow R$.
This ensures that $\cR$ admits quantifier elimination.
Let $\mathbb{K} = (K,+,\times,\ldots)$ be an elementary extension of $\cR$.
If for every $b \in K$ the set $\{ a \in R: a \leqslant b \}$ has a supremum in $R \cup \{ -\infty, \infty \}$ then we say that $\mathbb{K}$ is \textbf{tame extension} of $\cR$ and we say that $\cR$ is a \textbf{tame substructure} of $\mathbb{K}$.
We assume that all tame extension are nontrivial extensions, so as not to consider $\cR$ to be a tame elementary extension of itself.
If $\mathbb{K}$ is a tame elementary extension of $\cR$ then we define a standard part map $\std: \mathbb{K} \cup \{-\infty,\infty\} \rightarrow \cR \cup \{-\infty, \infty\}$ by
$$ \std(b) = \sup\{ a \in R : a < b\} .$$
The study of tame extensions began with the following theorem \cite{Marker-Steinhorn}:
\begin{thm}[Marker-Steinhorn]
The following are equivalent:
\begin{enumerate}
\item $\mathbb{K}$ is a tame extension of $\cR$,
\item Every type over $\cR$ realized in $\mathbb{K}$ is definable over $\cR$.
\end{enumerate}
\end{thm}
There are two particularly important kinds of tame extensions with which we are concerned.
First, if $\cR$ expands the ordered field of real numbers than every elementary extension of $\cR$ is tame.
Second, let $\zeta$ be a positive element of an elementary extension of $\cR$ which is less then every positive element of $\cR$.
Then $\cR(\zeta)$, the prime model over $\zeta$, is a tame elementary extension of $\cR$.
Indeed, as $\cR$ admits definable Skolem functions, every element of $\cR(\zeta)$ is of the form $f^*(\zeta)$ for some $\cR$-definable function $f: R^> \rightarrow R$ and in this case:
$$ \std f^*(\zeta) = \lim_{ t \rightarrow 0^+} f(t). $$
A \textbf{tame pair} is the expansion of $\mathbb{K}$ by a predicate defining $R$, denoted by $(\mathbb{K}, R)$.
The theory of tame pairs is complete:

\begin{thm}[van den Dries - Lewenberg]
Let $\mathbb{K},\mathbb{K}'$ be tame extensions of $\cR$.
Then $(\mathbb{K},R)$ and $(\mathbb{K}',R)$ are elementarily equivalent.
\end{thm}

Recall that the convex hull of a subfield of an ordered field is a subring, and that any convex subring of an ordered field is a valuation subring.
We say that a convex subring $\mathcal{O} \subseteq K$ is \textbf{T-convex} if it is the convex hull of a tame substructure of $\mathbb{K}$.
For the remainder of this section $\mathcal{O}$ is the convex hull of $R$ in $K$.
The expansion of $\mathbb{K}$ by a predicate defining $\mathcal{O}$ is called a \textbf{T-convex structure} and denoted $(\mathbb{K},\mathcal{O})$.
Let $\mathfrak{m}$ be the maximal ideal of $\mathcal{O}$ and let $\Gamma$ be the value group of the associated valuation.
It follows from the definition of a tame expansion that $b - \std(b) \in \mathfrak{m}$ for all $b \in \mathcal{O}$.
We therefore identify the residue field of $(\mathbb{K}, \mathcal{O})$ with $R$ and the residue map with $\std$.

\begin{thm}[van den Dries - Lewenberg]\label{thm:tconbassss}
The structure $(\mathbb{K},\mathcal{O})$ admits elimination of quantifiers and is universally axiomatizable.
\end{thm}
The following corollary is proven in \cite{vddii}.
The proof uses the Marker-Steinhorn theorem.
\begin{cor}
Every $(\mathbb{K},\mathcal{O})$-definable subset of $R^n$ is $\cR$-definable.
\end{cor}
In model-theoretic terminology, $\cR$ is \textbf{stably embedded} in $(\mathbb{K},\mathcal{O})$.
In particular, if $A \subseteq \mathcal{O}^n$ is $(\mathbb{K},\mathcal{O})$-definable then $\std(A) \subseteq R^n$ is $\cR$-definable.

\subsection{Uniform Limits}
Following \cite{vdDLimit} we apply stable embeddedness to show definibility of pointwise and uniform limits of definable families of functions.
We use this to construct the definable completion of a definable metric space in Section~\ref{section:completion}.

\begin{lem}\label{lem:familyofuniform}
Let $A \subseteq R^k$ be definable and let $\mathcal{F}=\{f_x : x\in R^l\}$ be a definable family of functions $A\rightarrow R$.
There is a definable family $\mathcal{G}$ of functions $A \rightarrow R$ such that a function $g: A \rightarrow R$ is an element of $\mathcal{G}$ if and only if there is a path $\gamma$ in $R^l$ such that $g$ is the pointwise limit of $f_{ \gamma(t) }$ as $t \rightarrow 0^+$.

If $\cR$ expands the real field then $\cG$ is the set of pointwise limits of sequences of elements of $\mathcal{F}$.
\end{lem}
This is essentially proven by van den Dries in \cite{vdDLimit}, however, he only worked over the reals.
We only sketch the proof.
\begin{proof}[Sketch]
Let $\cR^*=(R^*,+,\times,\leqslant,\ldots)$ be a tame elementary extension of $\cR$. 
We consider the tame pair $(\cR^*, R)$.
Stable embeddedness implies that there is an $\cR$-definable family $\cG=\{g_x:x\in B\}$ of functions $A\rightarrow R$ such that a function $g: A \rightarrow R$ is an element of $\cG$ if and only if for some $\alpha \in (R^*)^l$ we have:
$$g(a) = \std f^*_{\alpha}(a) \quad \text{for all } a \in A. $$
The completeness of the theory of tame pairs implies that $\cG$ does not depend on the choice of $\cR^*$.
Suppose $\cR^*=\cR(\zeta)$ is the prime model over a positive infinitesimal element $\zeta$ of an elementary extension of $\cR$.
Then the elements of $\cG$ are functions of the form $\std f^*_{\gamma^*(\zeta)}: A \rightarrow R$ for $\cR$-definable paths $\gamma$ in $R^l$.
For any path $\gamma$ in $R^l$:
$$\std f^*_{\gamma^*(\zeta)}(a) = \lim_{t\rightarrow 0^+} f_{\gamma(t)}(a) \quad \text{for all } a \in A. $$
Suppose that $\cR$ expands the real field.
Then every element of $\cG$ is a pointwise limit of a sequence of the form 
$$\{ f_{\gamma(\frac{1}{i})} \}_{ i \in \mathbb{N}} \quad \text{for a path } \gamma \text{ in } \bR^l.$$
We show that every pointwise limit of a sequence of elements of $\mathcal{F}$ is an element of $\cG$.
Let $\mathcal{U}$ be a nonprincipal ultrafilter on $\mathbb{N}$ and let $\cR^\mathcal{U} = (\bR^\mathcal{U},+, \times, \ldots)$ be the ultrapower of $\cR$ with respect to $\mathcal{U}$.
Let $\{ \alpha(i)\}_{ i \in \mathbb{N} }$ be a sequence of elements of $\bR^l$ such that $\{ f_{\alpha(i)} \}_{ i \in \mathbb{N} }$ pointwise converges as $i \rightarrow \infty$.
We let $\alpha$ be the element of $(\bR^\mathcal{U})^l$ corresponding to the sequence $\{ \alpha_i \}_{ i \in \mathbb{N} }$.
Then
$$ \std f^*_{\alpha}(a) = \lim_{ i \rightarrow \infty } f_{\alpha(i)}(a) \quad \text{for all } a \in A.$$
\end{proof}

In Proposition~\ref{prp:defcomplenion} we use the following corollary to construct the definable completion of a definable metric space.

\begin{cor}\label{cor:defuniformlimits}
There is a definable family $\cG'$ of functions $A\rightarrow R$ such that a function $g$ is an element of $\cG'$ if and only if it is a uniform limit of a family $\{f_{\gamma(t)}\}_{ t \in R^>}$ as $t\rightarrow 0^+$ for some definable $\gamma:R^{>}\rightarrow R^l$.
If $\cR$ expands the ordered field of reals then $\cG'$ is the set of uniform limits of sequences of elements of $\mathcal{F}$.
\end{cor}

\begin{proof}[Sketch]
We let $\cG'$ be the set of $g\in \cG$ such that for every $\epsilon\in\R^{>}$ there is an $x\in R^l$ such that 
$\|g-f_x \|_{\infty} \leqslant \epsilon$.
It is easy to check that this works.
\end{proof}

\section{Examples: Metrics Associated To Definable Families}\label{section:families}

In this section we describe definable metric spaces associated to definable families of functions and sets.

\subsection{Hausdorff Metrics}
Let $\mathcal{A} = \{ A_x  : x \in R^l \}$ be a definable family of subsets of $R^k$ such that $d_{H}(A_x, A_y) < \infty$ for all $x,y \in R^l$.
We put a definable pseudometric $d$ on $R^l$ by declaring
$$ d(x,y) = d_{H}(A_x, A_y)\quad\text{for all }x,y \in R^l.$$

 \subsection{Metrics Associated To Definable Families of Functions}

Let $\mathcal{F} = \{ f_x : x \in R^l \}$ be a definable family of functions $R^k \rightarrow R$.
If each element of $\mathcal{F}$ is bounded then we put a uniform pseudometric $d_\infty$ on $R^l$ by declaring:
$$ d_\infty(x,y) = \| f_x - f_y\|_\infty = \sup \{ | f_x(a) - f_y(a) | : a \in R^k \}\quad\text{for all }x,y\in R^l.$$
If every element of $\mathcal{F}$ is $C^r$ with bounded $r$th derivative then
$$ d(x,y) = \max_{ 0 \leqslant i \leqslant r } \{ \| f^{(i)}_x - f^{(i)}_y \|_{\infty} \} $$
is a definable pseudometric on $R^l$.
Thomas considered definable metric spaces of this form in \cite{Thomas}.

\subsection{Examples constructed using the CLR-volume Theorem}

Let $\mu_m$ be the $m$-dimensional Lebesgue measure on $\bR^k$.
Comte, Lion and Rolin proved the following proposition in \cite{CLR}.

\begin{prp}
Let $\{ A_x : x \in \bR^l\}$ be an $\ran$-definable family of subsets of $\bR^k$.
Then $\mu_m(A_x)$ is an $\ranex$-definable function $\bR^l \rightarrow \bR\cup\{\infty\}$.
\end{prp}
This implies that if $f: \bR^l \times \bR^k \rightarrow \bR$ is $\ran$-definable then
$$ \int_{ \bR^k } f(a, x) \; dx $$
is an $\ranex$-definable, $\bR\cup\{\infty\}$-valued function of $a \in \bR^l$.
Let $\mathcal{F} = \{ f_x  : x \in \bR^l \}$ be an $\ran$-definable family of functions $\bR^k \rightarrow \bR$.
If every element of $\mathcal{F}$ is $L^p$ integrable then
$$ d_p(x,y) = \left[ \int_{ \bR^k } | f_x(a) - f_y(a) |^p \; d\mu_k \right]^{ \frac{1}{p} } $$ 
is an $\ranex$-definable pseudometric on $\bR^l$.
If $\{ A_x : x \in \bR^l \}$ is an $\ran$-definable family of bounded $m$-dimensional subsets of $\bR^k$ then 
$$ d_\mu(x,y) = \mu_m [ A_x \Delta A_y ] $$
is an $\ranex$-definable pseudometric on $\bR^l$.

\section{Definable Metrics with Noneuclidean Topology}\label{section:noeu}
\textit{In this section $\mathcal R$ is an expansion of the real field.}
We give examples of definable metric spaces which are not homeomorphic to any definable set equipped with its euclidean topology.
The trivial example is any infinite definable set $X$ equipped with the discrete metric $d_{disc}$ given by declaring $d_{disc}(x,x') = 1$ whenever $x \neq x'$.
We now give an example of a definable metric space which is not locally contractible.
Let $d$ be the metric on $[0,1]$ given by $d(x,y)=0$ when $x=y$ and $d(x,y)=\max\{x,y\}$ otherwise.
It is clear that $d$ is symmetric and reflective.
The ultrametric triangle inequality holds as:
$$ \max\{ x,y \} \leqslant \max\{ \max\{ x,z\}, \max\{y,z\} \} = \max\{ d(x,z), d(y,z) \} \quad \text{for all } x,y,z \in X. $$ 
Every point in $[0,1]$ is isolated except for $0$, which is in the closure of $(0,1)$.
This metric space is not locally contractible at $0$.

\subsection{Metric Spaces Associated to Definable Simplicial Complexes}\label{subsection:graph}
We assume that the reader is familiar with simplicial complexes.
Our source is the classic \cite{Spanier}.
An \textbf{abstract simplicial complex} $\mathcal{V} = (V, \mathcal{C})$ is a set $V$ of vertices together with a collection $\mathcal{C}$ of finite subsets of $V$ such that:
\begin{enumerate}
\item Every singleton subset of $V$ is an element of $\mathcal{C}$,
\item Every subset of an element of $\mathcal{C}$ is an element of $\mathcal{C}$.
\end{enumerate}
We say that $\mathcal{V}$ is \textbf{finite-dimensional} if there is an $N \in \mathbb{N}$ such that every element of $\mathcal{C}$ has cardinality at most $N$.
We only consider finite-dimensional simplicial complices.
We say that $\mathcal{V}$ is \textbf{locally finite} if every element of $\mathcal{C}$ intersects only finitely many elements of $\mathcal{C}$.
A \textbf{definable abstract simplicial complex} is a definable set $V$ together with a definable family $\mathcal{C} = \{ C_x : x \in \bR^l \}$ of finite subsets of $V$ such that $(V,\mathcal{C})$ is an abstract simplicial complex.
We associate a topological space $|\mathcal{V}|$ to an abstract simplicial complex $\mathcal{V}$ called the \textbf{geometric realization} of $\mathcal{V}$.
If $\mathcal{V}$ is locally finite then the construction we give is equivalent to any other construction of the geometric realization that the reader may have seen.
We let $|\mathcal{V}|$ be the set of functions $\alpha:V\rightarrow [0,1]$ such that:
$$\{v\in V:\alpha(v)\neq 0 \}\in\mathcal{C} \quad \text{and} \quad \sum_{v\in V} \alpha(v)=1. $$
We put a metric $d_\mathcal{V}$ on $|\mathcal{V}|$ by declaring:
$$d_\mathcal{V}(\alpha,\beta)= \sqrt{\sum_{v\in V}|\alpha(v)-\beta(v)|^2}.  $$
The following fact is clear:
\begin{fact}
If $\mathcal{V}$ is a definable abstract simplicial complex then $(|\mathcal{V}|,d_\mathcal{V})$ is a definable metric space.
\end{fact}
This construction produces interesting definable metric spaces.

\subsection{Cayley Graphs}
Let $G$ be a finitely generated group.
We say that $G$ is \textbf{definably representable} if there is a free action of $G$ by definable functions on a definable set $A$.
Any finitely generated subgroup of $Gl_n(\bR)$ is definably representable.
Suppose that $G$ is a finitely generated group with a fixed symmetric set of generators $S = \{g_1, \ldots, g_n\}$.
Let $G \curvearrowright A$ be a free action of $G$ on a definable set $A$.
We define the \textbf{Cayley graph} of this action.
This depends on the choice of $S$.
We consider the elements of $S$ as functions $A\rightarrow A$.
The Cayley graph of $G \curvearrowright A$ is the graph $\mathcal{G}$ with vertex set $A$ and edge set
$$E=\{(a,b)\in A:(\exists 1 \leqslant  i \leqslant n) g_i(a)=b\}.$$
Then $\mathcal{G}$ is a definable graph, and the geometric realization of $\mathcal{G}$ is a definable metric space.
As a graph, $\mathcal{G}$ is the disjoint union of continuum many copies of the usual Cayley graph of $G$.
In this manner we construct definable metric spaces with unexpected properties.

\begin{example}[Sketch]
We give an example of two homeomorphic definable metric spaces $(X,d),(X',d')$ such that if $\tau: (X,d) \rightarrow (X',d')$ is a homeomorphism then $\tau$ induces a non Lebesgue measurable map.
We let:
$$ (X,d) = (\mathbb{R},e)\times (\mathbb{R},d_{\di}). $$
Let $\lambda \in (0,1)$ be an irrational number.
Let $\theta: [0,1] \rightarrow [0,1]$ be given by $\theta(t) = t + \lambda$ when $t + \lambda \in [0,1]$ and $g(t) = t + \lambda - 1$ otherwise.
Consider the action of $(\mathbb{Z}, +)$ on $[0,1]$ generated by $\theta$.
Let $\mathcal{G}$ be the Cayley graph of this action and let $(X',d')$ be the geometric realization of $\mathcal{G}$.
As a topological space, $(X',d')$ is the disjoint union of continuum many copies of $(\bR,e)$, so $(X',d')$ is homeomorphic to $(X,d)$.
We put an equivalence relation $\sim$ on $X'$ by declaring $x \sim y$ when $x$ and $y$ lie in the same connected component of $(X',d')$.
Recall that $[0,1]$ is the set of vertices of $\mathcal{G}$.
Two elements of $[0,1]$ are $\sim$-equivalent if and only if they lie in the same orbit of the action of $(\mathbb{Z},+)$.
The restriction of $\sim$ is essentially the Vitali equivalence relation.
Suppose that $h: (X',d') \rightarrow (X,d)$ is a homeomorphism.
Two elements $x,y \in X$ are in the same connected component of $(X,d)$ if and only if $h(x)$ and $h(y)$ lie in the same connected component of $(X,d)$.
Let $\tau: [0,1] \rightarrow \bR$ be the restriction of $\pi \circ h$ to $[0,1]$.
If $x,y \in [0,1]$ then $x \sim y$ holds if and only if $\tau(x)=\tau(y)$.
It is well-known that such a map $\tau: [0,1] \rightarrow \bR$ cannot be Lebesgue measurable.
\end{example}

\begin{example}
A definable metric space $(Y,d)$ is \textbf{definably connected} if $\emptyset$ and $Y$ are the only definable clopen subsets of $Y$.
Let $(X',d')$ be the definable metric space constructed in the previous example.
We show that $(X',d')$ is definably connected.
This gives an example of a definably connected metric space which is not connected.
Suppose that $A_1,A_2 \subseteq X'$ are disjoint nonempty clopen definable subsets of $X$.
Fixing $p_1 \in A_1, p_2 \in A_2$ we let $B_1,B_2 \subseteq [0,1]$ be the set of vertices which lie in the connected component of $p_1,p_2$, respectively.
As $B_1$ and $B_2$ are orbits of the action of $(\mathbb{Z}, +)$ on $[0,1]$, both $B_1$ and $B_2$ are dense in $([0,1],e)$.
As $B_1$ and $B_2$ are definable this implies $B_1 \cap B_2 \neq \emptyset$, contradiction.
\end{example}
In the next example we use a simple topological fact.
We leave the proof to the reader.
\begin{fact}\label{fact:graph}
Let $\mathcal{G}_1, \mathcal{G}_2$ be graphs where every vertex has degree at least $3$.
If the geometric realizations $|\mathcal{G}_1|$ and $|\mathcal{G}_2|$ are homeomorphic then $\mathcal{G}_1$ and $\mathcal{G}_2$ are isomorphic as graphs.
\end{fact}

\begin{example}[Sketch]\label{example:a}
We show that the Trivialization Theorem does not hold for definable families of metric spaces.
We construct a definable family of metric spaces which contains infinitely many elements up to homeomorphism.
Let $S = \{ x \in \bR^2 : \| x \| = 1 \}$, let $U(1)$ be the group of rotations of $S$ and let $\tau$ be the reflection across the $x$-axis.
If $\sigma\in U(1)$ then $\{\sigma,\tau\}$ generate a definable action of a dihedral group $D_\sigma$ on $S$.
The group $D_\sigma$ is finite if and only if $\sigma$ is a rational rotation.
Note that the Cayley graph of this action is $3$-regular.
Given $\sigma \in U(1)$ we let $X_\sigma$ be the geometric realization of the Cayley graph of this action.
Then $\{ X_\sigma : \sigma \in U(1) \}$ is a semialgebraic family of metric spaces.
Every connected component of $X_\sigma$ is homeomorphic to the geometric realization of the Cayley graph of $D_\sigma$.
Suppose that $\sigma,\eta \in U(1)$ and that $h: X_{\sigma}\rightarrow X_{\eta}$ is a homeomorphism.
Then $h$ maps connected components to connected components and thus induces a homeomorpism between the geometric realizations of the Cayley graphs of $D_\sigma$ and $D_\eta$.
Fact~\ref{fact:graph} implies that the Calyey graphs of $D_\sigma$ and $D_\eta$ are isomorphic as graphs.
Thus, if $\sigma$ and $\eta$ are rational rotations such that $|D_\sigma| \neq |D_\eta|$ then $X_\sigma$ is not homeomorphic to $X_\eta$.
So the family $\{ X_\sigma : \sigma \in U(1) \}$ contains infinitely many pairwise non-homeomorphic elements.
\end{example}

\section{Definable Completeness And Compactness}\label{section:completeness}
In this section we construct the definable completion of a definable metric space.
We show that definable completeness is equivalent to completeness over expansions of the real field.
Over expansions of the real field the definable completion agrees with the usual completion.
The notion of definable completeness was first studied in the context of definable metrics on definable families of functions by Thomas~\cite{Thomas}.
We also introduce a notion of definable compactness and show that over an expansion of the real field, definable compactness is equivalent to compactness.
\textit{Throughout this chapter $(X,d)$ and $(X',d')$ are definable metric spaces.}
A path $\gamma$ in $X$ is \textbf{Cauchy} if for every $\epsilon>0$ there is a $\delta>0$ such that if $0<t,t'<\delta$
then $d(\gamma(t),\gamma(t'))<\epsilon$.
An application of the triangle inequality proves that converging paths are Cauchy.
A path $\gamma$ in $X$ is said to be bounded if the image of $R^>$ under $\gamma$ is a bounded subset of $(X,d)$.
If the metric is not clear from context we will say that $\gamma$ $d$-converges.
A definable metric space is \textbf{definably complete} if every Cauchy path in it converges.

A definable metric space is \textbf{definably compact} if every path in it converges.
A definable metric space is \textbf{definably proper} if every bounded path in it converges.
A definably proper metric space is definably compact if and only if it is bounded.
A closed subset of a definably proper definable metric space endowed with the induced metric is definably proper and a definably proper space is definably complete.

\subsection{Basic Facts}
\begin{lem}\label{lem:continuity}
Let $h:(X,d)\rightarrow (X',d')$ be definable.
The following are equivalent:
\begin{enumerate}
\item $h$ is continuous. 
\item If $\gamma$ is a path in $(X,d)$ which converges to $x\in X$ then $h\circ \gamma$ converges to $h(x)$.
\end{enumerate}
\end{lem}

\begin{proof}
It is quite obvious that $(i)$ implies $(ii)$.
We prove the other implication.
Suppose that $h$ is not continuous at $x\in X$ and let $h(x) = x'$.
There exists a $\delta>0$ such that for every $\epsilon>0$ there is a $y\in X$ satisfying
$d(x,y)<\epsilon$ and $d'(h(y),x')>\delta$.
Applying definable choice there is a path $\gamma:\gR\rightarrow (X,d)$ such that $d(\gamma(t),x)<t$ and
$d'(h(\gamma(t)),x')>\delta$ holds for all $t > 0$.
\end{proof}

\begin{lem}\label{lem:closureincomplete}
Let $A\subseteq X$ be definable.
\begin{enumerate}
\item If $(X,d)$ is definably complete then $(A,d)$ is definably complete if and only if $A$ is closed.
\item If $(X,d)$ is definably compact then $(A,d)$ is definably compact if and only if $A$ is closed.
\item If $(X,d)$ is definable compact then $(A,d)$ is locally definably compact if and only if $A$ is locally closed.
\end{enumerate}
\end{lem}

\begin{proof}
We only prove the first claim as the proofs of the other two claims are very similar and the idea is familiar from metric space topology.
Suppose that $(X,d)$ is definably complete and that $A$ is closed.
Every Cauchy path in $A$ has a limit in $X$, and as $A$ is closed the limit must be an element of $A$.
Conversely, suppose that $A$ is not closed.
Let $p$ be a frontier point of $A$.
Applying definable choice we let $\gamma$ be a path in $A$ such that $d(\gamma(t),p)<t$ holds for all $t > 0$.
Then $\gamma$ is Cauchy with limit $p\notin A$.
So $(A,d)$ is not definably complete.
\end{proof}

\begin{prp}\label{prp:compactmax}
Suppose that $(X,d)$ is definably compact, and let $f: (X,d) \rightarrow (R,e)$ be definable and continuous.
Then $f$ has a maximum and a minimum.
\end{prp}

\begin{proof}
We prove that $f$ has a maximum.
Let $r=\sup\{(f(x):x\in X\}\in R_\infty$.
Let $\gamma:R^{>}\rightarrow X$ be a path in $X$ such that $0<r-(f\circ \gamma)(t)<r-t$ holds for all $t \in R^>$ if $r<\infty$ and $(f\circ \gamma)(t)>t$ holds for all $t \in R^>$ if $r=\infty$.
As $(X,d)$ is definably compact $\gamma$ must converge to some point $z$.
Continuity of $f$ implies $f(z)=r$.
This implies that $r<\infty$.
\end{proof}

\begin{lem}\label{lem:compactproduct}
A product of two definably compact metric spaces is definably compact.
\end{lem}

\begin{proof}
Suppose $(X,d)$ and $(X',d')$ are definably compact.
Let $\pi: X\times X' \rightarrow X$ and $\pi': X \times X' \rightarrow X'$ be the coordinate projections.
Let $\gamma$ be  a path in $X \times X'$.
Let $\xi = \pi \circ \gamma$ and $\xi' = \pi' \circ \gamma$.
By definable compactness, $\xi$ converges to some $x \in X$ and $\xi'$ converges to some $x' \in X$.
It follows immediately that $\gamma = (\xi,\xi')$ converges to $(x,x')$.
\end{proof}

\begin{lem}\label{lem:compactnodiscrete}
Suppose that $(X,d)$ is definably compact and infinite.
If $0 < t \ll 1$ then there are $x,y \in X$ such that $d(x,y) = t$.
\end{lem}

\begin{proof}
One of the following holds:
\begin{enumerate}
\item If $0 < t \ll 1$ there there are $x,y \in X$ such that $d(x,y) = t$.
\item If $0 < t \ll 1$ then $d(x,y) \neq t$ for all $x,y \in X$.
\end{enumerate}
We assume that $(ii)$ holds and derive a contradiction.
Suppose $\epsilon > 0$ is such that if $0 < t < \epsilon$ then there do not exist $x,y \in X$ such that $d(x,y) = t$.
Then if $x,y \in X$ and $d(x,y) < \epsilon$ then $x = y$, so $(X,d)$ is discrete.
As $X$ is infinite there is an injective path $\gamma : R^> \rightarrow X$.
As $(X,d)$ is discrete, $\gamma$ does not have a limit as $t \rightarrow 0^+$, contradiction.
\end{proof}

\begin{prp}\label{prp:compactuniformcontinuity}
Let $f: (X,d) \rightarrow (X',d')$ be a definable continuous function.
If $(X,d)$ is definably compact then $f$ is uniformly continuous.
If $(X,d)$ is definably compact and $f$ is bijective then $f$ is a uniform equivalence and hence a homeomorphism.
\end{prp}

\begin{proof}
Let $f:(X,d)\rightarrow (X',d')$ be a definable continuous map and suppose that $(X,d)$ is definably compact.
The proposition follows immediately in the case that $X$ is finite, so we suppose that $X$ is infinite.
For $t \in R^>$ let
$$ A_t = \{(x,y) \in X^2 : d(x,y) = t\}. $$
Applying Lemma~\ref{lem:compactnodiscrete} let $\epsilon > 0$ be such that if $t < \epsilon$ then $A_t \neq \emptyset$.
Each $A_t$ is closed, so it follows from the previous lemma that each $A_t$ is a definably compact subset of $X^2$. 
Define $g:(0,\epsilon)\rightarrow R^{\geqslant}$ by
$$g(t)=\max\{d'(f(x),f(y)): (x,y) \in A_t \}.$$
 Proposition~\ref{prp:compactmax} shows that $g$ is defined.
We have $d'(f(x),f(y)) \leqslant g(d(x,y))$ so to show that $f$ is uniformly continuous it suffices to show that $\lim_{t\rightarrow 0^+}g(t)=0$.
Suppose otherwise towards a contradiction.
There are paths $\gamma_1,\gamma_2:(0,\epsilon)\rightarrow X$ such that $d(\gamma_1(t),\gamma_2(t)) = t$ and $$d((f(\gamma_1(t)),f(\gamma_2(t)))\geqslant\delta \quad \text{ for some } \delta > 0.$$
Let $x_1,x_2$ be the limits of $\gamma_1$ and $\gamma_2$, respectively.
Then $d(x_1,x_2)=0$ so $x_1=x_2$, but $d(f(x_1),f(x_2))\geqslant\delta$, contradiction.

We now suppose in addition that $f$ is a bijection and show that $f$ is a uniform equivalence.
Let $h: (0,\epsilon) \rightarrow R$ be the definable function given by
$$ h(t) = \min\{d'(f(x),f(y)):(x,y) \in A_t\}. $$
As $A_t$ is definably compact $h(t)$ is defined.
As $f$ is injective for all $t > 0$ we have $h(t) > 0$.
As $h(t) \leqslant g(t)$, $\lim_{t\rightarrow 0^+}h(t) = 0$.
We have $h(d(x,y)) \leqslant d'(f(x),f(y))$, so $f$ is a uniform equivalence.
\end{proof}
The next lemma is used in a crucial way to prove Theorem~\ref{main}.
\begin{lem}\label{lem:coveringdefcompacts}
Let $(Y,d)$ be a definably compact metric space.
Let $X$ be a definably compact subset of euclidean space which admits a definable continuous surjection $(X,e) \rightarrow (Y,d)$.
Then there is a definable set $X'$ and a definable homeomorphism $(X',e) \rightarrow (Y,d)$.
\end{lem}

\begin{proof}
Let $h : (X,e) \rightarrow (Y,d)$ be a definable continuous surjection.
Let $E\subseteq X^2$ be the kernel of $h$, i.e., the equivalence relation 
$$ E= \{(x,y)\in X^2:h(x)=h(y)\}.$$
We endow the set-theoretic quotient $X/E$ with the quotient topology.
Continuity of $h$ implies that $E$ is a closed, and hence definably compact, subset of $X^2$.
We apply 2.15 of \cite{vdD} to obtain a definable set $X'$ and a definable continuous map
$p:X\rightarrow X'$ such that the kernel of $p$ is $E$ and the induced bijection $p_E: X/E\rightarrow (X',e)$ is a homeomorphism.
As $p$ is continuous, $(X',e)$ is definably compact.
Let $h_E: (X',e) \rightarrow (Y,d)$ be the map induced by $h$.
Then $h_E$ is a continuous bijection between definably compact metric spaces.
Proposition~\ref{prp:compactuniformcontinuity} implies that $h_E$ is a homeomorphism.
\end{proof}
We now prove two lemmas about extending definable functions:
\begin{lem}\label{lem:firstextend}
Suppose $(X',d')$ is definably complete.
Let $A \subseteq X$ be definable and dense and let $f: (A,d) \rightarrow (X',d')$ be a uniformly continuous definable function.
Then $f$ admits a unique extension to a uniformly continuous definable function $(X,d) \rightarrow (X',d')$.
\end{lem}

\begin{proof}
We construct an extension; it is clear from the construction that the extension is unique.
Let $g:R^> \rightarrow R^>$ be a definable continuous function such that $\lim_{t\rightarrow 0^+} g(t) = 0$ and $d'(f(x),f(y)) \leqslant g(d(x,y))$ whenever $x,y \in A$.
Applying definable choice let $\psi: X \times R^{>} \rightarrow A$ be a definable function such for each $(x,t) \in X \times R^{>}$ we have $d(\psi_x(t),x) < t$.
Fix $x \in X$ for the moment.
For any $t,t' \in R^>$ we have
$$ d'(f(\psi_x(t)),f(\psi_x(t'))) \leqslant g(d(\psi_x(t)),\psi_x(t'))).$$
As $\psi_x$ is Cauchy this inequality implies that $f \circ \psi_x$ is Cauchy.
Definable completeness of $(X',d')$ implies that $f \circ \psi_x$ converges.
Therefore, for any $x \in X$ we may let $\hat{f}(x)$ be the limit of $f \circ \psi_x$ in $(X',d')$.
A computation shows that $\hat{f}$ is uniformly continuous:
$$ d'(\hat{f}(x),\hat{f}(y)) = \lim_{t\rightarrow 0^+} d'(f(\psi_x(t)),f(\psi_y(t))) \leqslant \lim_{t\rightarrow 0^+} g(d(\psi_x(t),\psi_y(t))) = g(d(x,y)). $$
\end{proof}
The proof of Lemma~\ref{lem:firstextend} also shows the following:

\begin{lem}\label{lem:whatuneed}
Suppose $(X,d)$ and $(X',d')$ are definably complete.
Let $A \subseteq X$ be definable and dense.
Suppose that $f: A \rightarrow X'$ is a distance preserving function.
Then $f$ admits a unique extension to a distance preserving definable function $(X,d) \rightarrow (X',d')$.
\end{lem}
Immediately from this we have:

\begin{lem}\label{lem:whatuneed1}
Let $(X,d)$ and $(X',d')$ be definably complete metric spaces.
Let $A \subseteq X$ and $A' \subseteq X'$ be definable and dense.
A definable isometry $(A,d) \rightarrow (A',d')$ admits a unique extension to a definable isometry $(X,d) \rightarrow (X',d')$.
\end{lem}

\section{The Definable Completion}\label{section:completion}
We construct the definable completion of a definable metric space.
Our construction is an application of the following fact:

\begin{fact}\label{fact:uniformcompletefamily}
Let $A \subseteq R^k$ be a definable set.
Let $\mathcal{F} = \{ f_x : x \in R^l \}$ be a definable family of functions $A \rightarrow R$ and $\mathcal{G} = \{ g_x  : x\in B \}$ be the definable family of functions $A \rightarrow R$ given by Corollary~\ref{cor:defuniformlimits}.
Let $d_\infty$ be the pseudometric on $B$ given by $d_\infty(x,y) = \| f_x - f_y \|_\infty$.
Then the metric space associated to $(B, d_\infty)$ is definably complete.
If $\cR$ expands the ordered field of reals then the metric space associated to $(B, d_\infty)$ is complete.
\end{fact}

Recall that $\mathcal{G}$ consists of those functions $A \rightarrow R$ which are uniform limits of families of the form $\{ f_{\gamma(t)}\}_{ t \in R^> }$ as $t \rightarrow 0^+$ for paths $\gamma$ in $R^l$.

\begin{proof}
We show that the metric space associated to $(B, d_\infty)$ is definably complete.
Fix a Cauchy path $\gamma : R^> \rightarrow (B, d_\infty)$.
As $\gamma$ is Cauchy, $f_{ \gamma(t) }$ uniformly converges as $t \rightarrow 0^+$.
Let $g_0: A \rightarrow R$ be the uniform limit of $g_{\gamma(t)}$ as $t \rightarrow 0^+$.
Applying definable choice we let $\psi : B \times R^> \rightarrow R^l$ be a definable function such that:
$$ \| f_{\psi(x,t)} - g_x \|_\infty \leqslant t \quad \text{for all } x \in B, t \in R^>. $$
The triangle inequality implies:
$$ \| f_{ \psi( \gamma(t),t)} - g_0 \|_{\infty} \leqslant \| f_{ \psi( \gamma(t),t)} - g_{\gamma(t)} \|_{\infty} + \| g_{\gamma(t)} - g_0 \|_{\infty} \leqslant t + \| g_{\gamma(t)} - g_0 \|_{\infty}. $$
Thus $f_\psi(\gamma(t),t)$ uniformly converges to $g_0$ as $t \rightarrow 0^+$.
Thus $g_0$ is an element of $\mathcal{G}$.
If $\cR$ is an expansion of the ordered field of real numbers then a similar argument with sequences in place of paths shows that the metric space associated to $(B,d_\infty)$ is complete.
\end{proof}
For our definable metric space $(X,d)$ we construct a definably complete metric space $(\widetilde{X},\widetilde{d})$ and a distance preserving definable map $(X,d) \rightarrow (\widetilde{X},\widetilde{d})$ with dense image.
It follows from Lemma~\ref{lem:firstextend} that the definable completion is  left adjoint to the forgetful functor from the category of definably complete metric spaces and uniformly continuous definable maps to the category of all definable metric spaces and and uniformly continuous definable maps.
It follows from Lemma~\ref{lem:whatuneed1} that the definable completion is unique up to definable isometry.

\begin{prp}\label{prp:defcomplenion}
There is a definably complete metric space $(\widetilde{X},\widetilde{d})$ and a definable isometry $(X,d)\rightarrow(\widetilde{X},\widetilde{d})$ with dense image.
Any uniformly continuous definable map $(X,d)\rightarrow (X,d')$ admits a unique extension to a uniformly continuous definable map between the definable completions $(\widetilde{X},\widetilde{d})\rightarrow (\widetilde{X'},\widetilde{d'})$.
\end{prp}

\begin{proof}
We construct the definable completion.
We use the Kuratowski Embedding.
Fix a basepoint $p\in X$.
For $x\in X$ let $d_x:X\rightarrow R$ be given by $d_x(y)=d(x,y)$.
For each $x \in X$ consider the definable function $\tau_x:X\rightarrow R$ where $\tau_x(y)=d_x(y)-d_p(y)$.
Each $\tau_x$ is bounded as we have
 $$|d_x(y)-d_p(y)|=|d(x,y)-d(p,y)|\leqslant d(x,p) \quad \text{for all } y \in X.$$
We show that $\|\tau_{x}-\tau_{y}\|_{\infty}=d(x,y)$ for all $x,y \in X$.
Fix $x,y \in X$.
As $$|\tau_x(z)-\tau_{y}(z)|=|d_x(z)-d_{y}(z)|\leqslant d(x,y) \quad \text{for all } z \in X,$$
we have $\|\tau_x-\tau_{y}\|_{\infty}\leqslant d(x,y)$. 
As $\tau_x(x)=-d(x,p)$ and $\tau_{y}(x)=d(x,y)-d(x,p)$ we have $\|\tau_x-\tau_{y}\|_{\infty}=d(x,y)$.
Applying Corollary~\ref{cor:defuniformlimits} let $\cG=\{g_x:X\rightarrow R:x\in B\}$ be a definable family of functions whose elements are the uniform limit points of the family $\{\tau_x: x\in X\}$.
For $a,b\in B$ let $\widetilde{d}(a,b)=\|g_a-g_{b}\|_{\infty}$.
Let $(\widetilde{X},\widetilde{d})$ be the metric space associated to the pseudometric space $(B,\widetilde{d})$.
In Fact~\ref{fact:uniformcompletefamily} we observed that $(\widetilde{X},\widetilde{d})$ is definably complete.
The natural map $X\rightarrow \widetilde{X}$ is clearly injective, distance-preserving and has dense image.
The second claim follows from Lemma~\ref{lem:firstextend}.
\end{proof}

\begin{cor}\label{cor:defcompleteeqeomplete}
Suppose that $\cR$ expands the orderered field of reals.
Then every definably complete metric space is complete.
\end{cor}

\begin{proof}
Fact~\ref{fact:uniformcompletefamily} implies that the definable completion of any definable metric space is complete.
Lemma~\ref{lem:whatuneed1} implies that a definably complete metric space is isometric to its definable completion.
\end{proof}

We now prove the analogue of Corollary~\ref{cor:defcompleteeqeomplete} for definably compact spaces.
We need two simple lemmas.
We let $\mathcal{U}$ be a nonprincipal ultrafilter on the natural numbers and let $\cR^*=(R^{\mathcal{U}},+,\times,...)$ be the corresponding ultrapower of $\cR$.
Given a sequence of elements $\{x_i\}_{i\in\mathbb{N}}$ of some $\cR$-definable set $A$ we let $[x_i]$ be the corresponding element of $A^*$.

\begin{lem}\label{lem:nonstan1}
Suppose that $\cR$ expands the ordered field of reals.
Let $\{x_i\}_ {i\in\mathbb{N}}$ be a sequence of elements of $X$ and $y\in X$ be such that $\std d^*([x_i],y)=0$.
Then there is a subsequence of $\{x_i\}_{i \in \mathbb{N}}$ which converges to $y$.
\end{lem}

\begin{proof}
Let $\epsilon\in\bR^{>}$. 
As $d^*([x_i],y)<\epsilon$ we see that $d(x_i,y)<\epsilon$ for $\mathcal{U}$-many $i$.
 So for every such $\epsilon$ there is an $i \in \mathbb{N}$ such that $d(x_i,y)<\epsilon$.
 There is a subsequence of $\{x_i\}_{i \in \mathbb{N}}$ which converges to $y$.
\end{proof}

\begin{lem}\label{lem:nonstan2}
Let $\zeta$ be an infinitesimal positive element of an elementary extension of $\cR$ and 
let $\cR(\zeta)$ be the prime model over $\zeta$.
Let $\gamma$ be a path in $X$.
Then $\gamma$ converges to $y\in X$ if and only if $\std d^*(\gamma(\zeta),y)=0$.
\end{lem}

\begin{proof}
$\std d^*(\gamma(\zeta),y)=\lim_{t\rightarrow 0^+} d(\gamma(t),y)$.
\end{proof}

\begin{prp}\label{prp:coincidencecompacts}
Suppose $\cR$ expands the ordered field of reals.
Then every definably compact metric space is compact.
\end{prp}

\begin{proof}
It is clear that a compact definable metric space is definably compact.
Let $(X,d)$ be a definably compact metric space.
Consider the following sentence in the language of tame pairs:
$$\Theta=(\forall x\in X^{*})(\exists y\in X)[\std d^*(x,y)=0].$$
As $(X,d)$ is definably compact, Lemma~\ref{lem:nonstan2} implies that
$(\cR(\zeta),\cR)\models\Theta$.
As the theory of tame pairs is complete $(\cR^{\mathcal{U}},\cR)\models\Theta$.
By Lemma~\ref{lem:nonstan1} every sequence in $X$ admits a converging subsequence.
Thus $(X,d)$ is compact.
\end{proof}
The next lemma will be used in the proof of Proposition~\ref{prp:sepintocompact} below to embed certain definable metric spaces in definably compact metric spaces.

\begin{lem}\label{lem:properhau}
Let $\{ A_x: x \in R^l \}$ be a definable family of subsets of $[0,1]^k$.
Let $d$ be the pseudometric on $R^l$ given by $d(x,y) = d_H(A_x,A_y)$ for all $x,y \in R^l$.
Then the definable completion of the definable metric space associated to $(R^l,d)$ is definably compact.
\end{lem}

\begin{proof}
Let $(B,d)$ be the definable metric space associated to $(R^l,d)$ and let $\pi: R^l \rightarrow B$ be the quotient map.
Let $(\widetilde{B}, \widetilde{d})$ be the definable completion of $(B,d)$.
We show that an arbitrary path $\gamma: R^> \rightarrow \widetilde{B}$ converges.
Applying definable choice let $\gamma': R^> \rightarrow B$ be a path such that
$$ \widetilde{d}(\gamma'(t),\gamma(t)) \leqslant t \quad \text{ for all } t\in R^>. $$
It is enough to show that $\gamma'$ converges.
Let $\eta: R^> \rightarrow R^l$ be a path such that $\pi \circ \eta = \gamma'$.
Let $D \subseteq R^k$ be the set of $p$ such that $(0,p)$ is in the closure of
$$ \{ (t,q) : q \in A_{\eta(t)} \} \subseteq R^> \times R^k.$$
Lemma~\ref{lem:dez} implies that $A_{\eta(t)}$ converges to $D$ in the Hausdorff metric as $t \rightarrow 0^+$.
It follows that $\gamma'$ has a limit in $\widetilde{B}$.
\end{proof}

\section{Main Proof}
\textit{In this section $(X,d)$ is a definable metric space}.
The main goal of this section is to prove the following:

\begin{thm}\label{main}
Exactly one of the following holds:
\begin{enumerate}
\item $(X,d)$ is not \textbf{definably separable}, i.e. there is an infinite definable $A\subseteq X$ such that $(A,d)$ is discrete.
\item There is a definable set $Z$ and a definable homeomorphism
$$ (X,d) \rightarrow (Z,e). $$
\end{enumerate}
\end{thm}

The proof splits into two parts.
First we show that every definably separable metric space is definably homeomorphic to a subspace of a definably compact metric space and then we show that every definably compact metric space is homeomorpic to a definable set equipped with its euclidean topology.
We embedd definably separable spaces in definably compact metric spaces by studying the pseudometric $d_H$ on $X$ given by declaring $d_H(x,y)$ to be the Hausdorff distance between $\gr(d_x)$ and $\gr(d_y)$, where we set $d_x(z) = d(x,z)$.
If $X$ is a bounded subset of euclidean space and $(X,d)$ is a bounded metric space then $d_H$ is a psueodmetric on $X$ and the completion of the metric space associated to $(X,d_H)$ is definably compact.
We show that every definably separable metric space is definably isometric to a definable metric space $(X,d)$ for which $d_H$ is a metric and $\id: (X,d) \rightarrow (X,d_H)$ is a homeomorphism.
This entails showing that any definably separable metric space $(X,d)$ admits a partition into definable sets on which the $d$-topology agrees with the metric topology.

After constructing the embedding we use the aforementioned piecewise result and Aschenbrenner and Thamrongtanyalak's definable Micheal Selection Theorem to show that every definably compact metric space is a definable continuous image of a definably compact set.
An application of Lemma~\ref{lem:coveringdefcompacts} then shows that every definably compact metric space is homeomorphic to a definably compact set.

Along the way we prove some general facts about the topology of definable metric spaces.
As a consequence we show that if $\cR$ expands the real field then the topological dimension of $(X,d)$ is the largest $k$ for which there is a definable continuous injection $([0,1]^k,e) \rightarrow (X,d)$.

\subsection{Definably Separable Metric Spaces}

We say that $(X,d)$ is \textbf{definably separable} if every $d$-discrete definable subset of $X$ is finite.
This terminology is justified:
If $\cR$ is an expansion of the ordered field of reals then every infinite definable set has cardinality $|\bR|$, so if $(X,d)$ is not definably separable then $X$ contains a discrete subset with cardinality $|\bR|$, which implies that $(X,d)$ is not separable.
Theorem~\ref{main} thus implies that a metric space definable in an o-minimal expansion of the ordered real field is separable if and only if it is definably separable.
Lemma~\ref{lem:c} implies:
\begin{lem}\label{lem:useforhauss}
Suppose that $(X,d)$ is not definably separable.
There is a $t>0$ and an infinite definable $A\subseteq X$ such that $d(x,x')>t$ for any distinct $x,x'\in A$.
\end{lem}
\begin{proof}
Let $B \subseteq X$ be an infinite definable set such that $(B,d)$ is discrete.
Applying definable choice let $f: B \to R^>$ be a definable function such that if $x, y \in B$ and $d(x,y) < f(x)$ then $x = y$.
By Lemma~\ref{lem:c} there is an infinite definable subset $A \subseteq B$ and an $t \in R^>$ such that if $x \in A$ then $f(x) > t$.
\end{proof}

For the next lemma, we define the \textbf{local dimension} of $(X,d)$ at $x\in X$ to be
$$ \dim_x(X,d) = \min \{ \dim B_d(x,t) : t>0 \}. $$
\begin{lem}\label{lem:firstsep}
If $(X,d)$ is definably separable then
$\dim_x(X,d)=\dim (X)$ at almost every $x\in X$.
\end{lem}

\begin{proof}
We prove the contrapositive.
To this effect we  an $e$-open $m$-dimensional definable $U \subseteq X$ such that $\dim_x(X,d) < \dim(X)$ for all $x\in U$.
After replacing $U$ with a smaller open set with the same properties if necessary we may also suppose that $\dim_x(X,d)=m<n$ for all $x\in U$.
Let $g:U\rightarrow R^{>}$ be a definable function such that $\dim [B_d(x,g(x))]=n$ for all $x \in U$.
We set $B_x=B_d(x,g(x))$.
We apply a uniform version of the good directions lemma see 4.3 in \cite{vdD}.
We find an $m$-dimensional open $W\subseteq U$ and a linear projection $\pi:U\rightarrow R^{n}$ whose restriction to $B_x$ is finite for each $x\in W$.
As $\dim \pi(W)=n$ the fiber lemma for o-minimal dimension implies that there is a $q\in W$ such that $\dim \pi^{-1}(q)$ is $(m-n)$-dimensional.
We fix such a $q$ and let $J=\pi^{-1}(q)$.
Then $J$ is infinite.
We show that $(J,d)$ is discrete.
Fix $\beta\in J$.
Let $\{\beta_1,\ldots,\beta_N\}=B_\beta\cap J$.
Then if $y\in J$ and $d(\beta,y)<d(\beta,\beta_i)$ holds for each $i$ then $y=\beta$, so $\beta$ is isolated in $J$.
\end{proof}
The next lemma shows that there is no definable compactification of a definable discrete metric space.
\begin{lem}\label{lem:compactissep}
If $(X,d)$ is definably compact and $B\subseteq X$ is definable then
$(B,d)$ is definably separable.
\end{lem}

\begin{proof}
We prove the contrapositive. 
Suppose that $(B,d)$ is not definably separable.
Suppose that $A\subseteq X$ and $t>0$ satisfy the conditions in Lemma~\ref{lem:useforhauss}.
A path in $A$ is Cauchy if and only if it is eventually constant.
Thus there is a path in $X$ which is not Cauchy and so $(X.d)$ is not definably compact.
\end{proof}

\subsection{A Definable Compactification}
In this section we show that every definably separable metric space is definably homeomorphic to a definable subspace of a definably compact metric space.
After this, it suffices to prove Theorem~\ref{main} for definably compact metric spaces.
Along the way, we will prove several other results about arbitrary definable metric spaces.

We introduce some terminology.
We say that a point $x\in X$ is \textbf{yellow} if there is a path in $X$ which converges to $x$ in the euclidean topology and converges in the $d$-topology to some $y\in X\setminus\{x\}$.
We say that a point $x\in X$ is \textbf{blue} if it is not yellow.
We say that definable metric space is \textbf{blue} if all of its points are blue.
By a ``blue metric space" we mean a ``blue definable metric space".
Our first goal is to show, in Proposition~\ref{prp:keytechnical}, that every definable metric space is definably isometric to a blue metric space.
We first show that the set of blue points of $(X,d)$ is definable.
For $y \in X$ let $d_y:X\rightarrow R$ be given by $d_y(x)=d(y,x)$.

\begin{lem}\label{lem:defofgoodpoint}
The following are equivalent for $x,y\in X$:
\begin{enumerate}
\item There is a path in $X$ which $e$-converges to $x$ and $d$-converges to $y$.
\item $(x,0)\in\cl_{e}[\gr(d_y)]$.
\end{enumerate}
It follows that the set of yellow points is definable.
\end{lem}

\begin{proof}
Suppose that $\gamma$ is a path in $X$ which $e$-converges to $x$ and $d$-converges to $y$.
Then $(\gamma(t),(d_y \circ \gamma)(t))$ is a path in $\gr(d_y)$ which $e$-converges to $(x,0)$.
Conversely, suppose that $(x,0)\in\cl_{e}[\gr(d_y)]$.
Let $\gamma$ be a path in $\gr(d_y)$ which $e$-converges to $(x,0)$.
Let $\pi_X:X\times R\rightarrow X$ and $\pi_R:X\times R\rightarrow R$ be the coordinate projections.
Then $\pi_1\circ \gamma$ converges to $x$ and $\pi_R\circ \gamma$ converges to $0$.
As $\pi_X \circ \gamma = d_y \circ \pi_X \circ \gamma$, $\pi_X \circ \gamma$ is a path in $X$ which $e$-converges to $x$ and $d$-converges to $y$.
\end{proof}
As in Section~\ref{section:completion} , $(\widetilde{X},d)$ is the definable completion of $(X,d)$.
\begin{lem}\label{lem:goddzz}
Let $x\in X$.
Then $x$ is a blue point of $(\widetilde{X},d)$ if and only if any $d$-Cauchy path in $X$ which $e$-converges to $x$ must also $d$-converge to $x$. 
\end{lem}

\begin{proof}
Suppose that $x$ is a blue point of $\widetilde{X}$ and that $\gamma$ is a $d$-Cauchy path which $e$-converges to $x$.
It follows that $\gamma$ $d$-converges in $\widetilde{X}$ and so $\gamma$ must $d$-converge to $x$.
We now prove the converse.
Suppose that every $d$-Cauchy path in $X$ which $e$-converges $x$ also $d$-converges to $x$.
Suppose that $\gamma$ is a path in $\widetilde{X}$ which $e$-converges to $x$.
Suppose that $\gamma$ $d$-converges to $y \in \widetilde{X}$.
Applying definable choice and the density of $X$ in $\widetilde{X}$ there is a path $\eta$ in $X$ such that $d(\gamma(t),\eta(t))<t$ for all $t \in R^>$.
Thus $\eta$ $d$-converges to $y$, so $\eta$ is $d$-Cauchy.
It follows from the assumption on $x$ that $\eta$ $d$-converges to $x$, so $\gamma$ $d$-converges to $x$.
\end{proof}

\begin{prp}\label{prp:keytechnical}
Every definable metric space is definably isometric to a blue definable metric space.
\end{prp}

To prove Proposition~\ref{prp:keytechnical} it suffices to prove the following claim:

\begin{clm}\label{clm:clm0}
Almost every $x \in X$ is blue.
\end{clm}

Before proving this claim we suppose that it on holds every definable metric space and prove Proposition~\ref{prp:keytechnical}.
We apply induction on $\dim(X)$.
If $\dim(X)=0$ then $(X,e)$ and $(X,d)$ are both discrete and so $(X,d)$ is blue.
Suppose $\dim(X)>0$.
We let $A$ be the definable set of points of $X$ which are yellow in $(\tilde{X},d)$.
Clearly $(X\setminus A,d)$ is blue and $\dim(A)<\dim(X)$.
Applying the inductive assumption on $(A,d)$ we produce a definable isometry $\tau:(A,d)\rightarrow (A',d')$ to a blue metric space $(A',d')$.
Let $X'$ be the disjoint union of $X\setminus A$ and $A'$. 
Let $\sigma:X \rightarrow X'$ be the natural bijection and let $d'$ be the pushfoward of $d$ by $\sigma$.
Any path in $X'$ is eventually contained in $A'$ or $X\setminus A$.
This implies that $(X',d')$ is blue.

We prove a slight strengthening of Claim~\ref{clm:clm0}, where we use the full strength of the hypothesis in the inductive step.

\begin{clm}\label{clm:clm1}
Almost every $x\in X$ is a blue point of $\widetilde{X}$.
\end{clm}

\begin{proof}
We apply induction on $\dim(X)$.
If $\dim(X)=0$ then both $(X,e)$ and $(X,d)$ are discrete and the claim trivially holds.
We assume that $\dim(X)>0$ and let $l = \dim(X)$.
Let $\sigma: X \rightarrow R^l$ be an injective definable map.
Let $X' = \sigma(X)$ and $d'$ be the metric on $X'$ given by $d'(\sigma(x),\sigma(y)) = d(x,y)$.
Almost every $p \in X$ has a neighborhood $U$ such that $\sigma|_U$ is a homeomorphism onto an open neighborhood of $\sigma(p)$, and if $p$ has such a neighborhood then $\sigma(p)$ is blue in $X'$ if and only if $p$ is blue in $X$.
It suffices to show that almost every point in $(X',d')$ is blue.

We therefore suppose without loss of generality that $X\subseteq R^l$.
We assume towards a contradiction that $\dim A=l$.
We apply definable choice to produce definable functions
$g:A\rightarrow \widetilde{X}$ and $\psi:A\times R^{>}\rightarrow X$ such that:
$$ g(x) \neq x \quad \text{and} \quad e(\psi_x(t),x)<t,  d(\psi_x(t),g(x))<t \quad \text{for all } x \in A, t \in R^>.$$
Then $\psi_x$ is a path in $X$ which $e$-converges to $x$ and $d$-converges to $g(x)$.
Let $\pi:X\rightarrow R$ be the projection onto the last coordinate.
Fix $y\in \pi(X)$.
We have $\dim \pi^{-1}(y) \leqslant l-1$ so we can apply the inductive assumption to $(\pi^{-1}(y),d)$.
For almost every $x\in \pi^{-1}(y)$ we have the following: a $d$-Cauchy path in $\pi^{-1}(y)$ which $e$-converges to $x$ also $d$-converges to $x$.
This holds for every $y\in \pi(X)$
so we apply the fiber lemma for o-minimal dimension and conclude that for almost every $x\in X$ we have the following: a $d$-Cauchy path $\gamma$ in $X$ which $e$-converges to $x$ and satisfies $(\pi\!\circ\!\gamma)(t)=\pi(x)$ when $0<t\ll 1$ also $d$-converges to $x$.
Let $A'\subseteq A$ be the definable set of points in $A$ which satisfy this condition.
Then $A'$ is $l$-dimensional.
If $x\in A'$ and $0<t\ll 1$ then $(\pi\!\circ\psi_x)(t)\neq \pi(x)$.
We let 
$$A'_{+}=\{x\in A':[0<t\ll 1]\longrightarrow [\pi(\psi_x(t)) > \pi(x)]\}$$
 and $$A'_{-}=\{x\in A':[0<t\ll 1]\longrightarrow [\pi(\psi_x(t)) < \pi(x)]\}.$$
As $A'=A'_{+}\cup A'_{-}$, at least one of $A'_{+}$ or $A'_{-}$ is $l$-dimensional. 
  We suppose without loss of generality that $\dim A_{+}=l$.
   Let $J\subseteq R$ be an interval and $B\subseteq R^{l-1}$ be a product of intervals such that $J\times B\subseteq A'_{+}$. 
   For every $x\in J\times B$ we have $\pi(\psi_x(t))>\pi(x)$ when $0<t\ll 1$.
 
 \begin{clm}\label{clm:nine}
 For all $b\in J$ there is an open $V\subseteq B$ and $\delta,s>0$ such that if $y\in (b,b+\delta)\times V$ then $d(\{b\}\times V,y)>s$.
 \end{clm}
 
 Fix $b\in J$.
 Suppose $x\in \{b\} \times B$.
 As $x\in A'$ there is a $s>0$ such that if $z\in X$ satisfies $d(g(x),z)<s$ and $e(x,z)<s$ then $\pi(z) \neq b$.
Applying Lemma~\ref{lem:c} there is an open $V'\subseteq B$ and $s>0$ such that if $x\in \{b\} \times V'$ and $z\in X$ satisfy $d(g(x),z)<s$ and $e(x,z)<s$ then $\pi(z)\neq b$.
If $x,z\in\{b\}\times V'$ and $e(x,z)<s$ then $d(g(x),z)\geqslant s$. 
By shrinking $V'$ if necessary we may assume that $\diam_e(V')<s$. 
Now if $x,z\in\{b\}\times V'$ then $d(g(x),z)>s$. 
This implies that if $x\in \{b\} \times V'$ and $0<t<s$ then $\pi(\psi_x(t))>b$.
Applying Lemma~\ref{lem:e} to the image $\psi$ we find an open $V\subseteq V'$ and a $\delta>0$ such that 
$$(b,b+\delta)\times V \subseteq \left\{\psi_x(t):x\in V', 0<t<\frac{1}{2}s\right\}.$$ 
We show that $V$ and $\delta$ satisfy the conditions of Claim~\ref{clm:nine}.
Let $y\in (b,b+\delta)\times V$ and $z\in \{b\}\times V$.
Fix $(x,t)\in V'\times (0,\frac{s}{2})$ such that $\psi_x(t)=y$.
We have $d(g(x),y)<\frac{s}{2}$ and $d(g(x),z)>s$ so the triangle inequality implies that $d(z,y)>\frac{1}{2}s$.
This proves Claim~\ref{clm:nine}.

\begin{clm}\label{clm:ten}
There is an interval $J'\subseteq J$, a product of intervals $B'\subseteq B$ and $s>0$ such that if $y,z\in J'\times B'$ and $\pi(y)\neq \pi(z)$ then $d(x,y)>s$.
\end{clm} 

Applying definable choice to Claim~\ref{clm:nine} we fix definable functions $f,g:J\rightarrow R$ and a definable set $V\subseteq J\times B$ such that each $V_b$ is open and for all $b\in J$, if $y\in (b,b+f(b))\times V_b$ then $d(\{b\} \times V_b,y)>g(b)$.
There is an interval $J'\subseteq J$, a product of intervals $B'\subseteq B$ and $s,\delta>0$ such that for all $b \in J':$
\begin{enumerate}
\item $f(b)>s $ and $g(b)>\delta$,
\item $B'\subseteq V_b$.
\end{enumerate}
Here $(i)$ is immediate. $(ii)$ is immediate from the fiberwise openness theorem Theorem 2.2 of \cite{vdD}.
By shrinking $J'$ if necessary we suppose that $\diam_e(J')<\delta$.
If $b\in J'$ then 
$$\{x\in J'\times B':\pi(x)>b\} \subseteq (b,b+\delta)\times V_b\subseteq (b,b+g(b))\times V_b.$$
Thus if $x,y\in J'\times B'$ and $\pi(y)>\pi(x)$ then $d(x,y)>g(\pi(x))>s$. 
This proves Claim~\ref{clm:ten}.
Fix $x\in J'\times B'$.
As we have $\pi(\psi_x(t))\neq \pi(x)$ when $0<t\ll 1$ and $\pi(\psi_x(t))\rightarrow \pi(x)$ as $t\rightarrow 0^+$ we have $\pi(\psi_x(t)) \neq \pi(\psi_x(t'))$ when $0<t<t' \ll 1$.
Thus $d(\psi_x(t),\psi_x(t'))>s$ when $0<t<t'\ll 1$.
This contradicts the fact that $\psi_x$ is $d$-Cauchy.
This contradiction establishes Claim~\ref{clm:clm1}.
\end{proof}

Every blue metric space admits a definable distance decreasing injection into a definably compact metric space.
We define the map $\iota_X$ in the case when $X$ is a bounded subset of euclidean space.
We indicate here how to extend this definition to the general case.
Let $X'$ be a definable bounded subset of euclidean space for which there is a definable homeomorphism $\tau: X \rightarrow X'$.  Let $d'$ be the pushforward of $d$ onto $X'$ by $\tau$.
We then take $\iota_X = \tau \circ \iota_{X'}$.
Then $\iota_X$ depends on the choice of $X'$ and $\tau$.
These choices will not matter, so we suppress them.
Note that, in the situation described, $(X',d')$ is blue if and only if $(X,d)$ is blue.
This ensures that $\iota_X$ is injective when $(X,d)$ is blue.

\begin{dfn}\label{dfn:iota}
Suppose that $X$ is a bounded subset of euclidean space.
Let $\mathbf{d}$ be the function on $X^2$ given by $\mathbf{d}(x,y) = \min\{ d(x,y),1\}$.
It is easily checked that $\mathbf{d}$ is a metric and $\id: (X,d) \rightarrow (X, \mathbf{d})$ is distance decreasing.
For each $x\in X$ we let $\mathbf{d}_x:X\rightarrow R$ be $\mathbf{d}_x(y)=\mathbf{d}(x,y)$.
We define a pseudometric $d_H$ on $X$ by setting $d_H(x,y)$ equal to the Hausdorff distance between $\gr(\mathbf{d}_x)$ and $\gr(\mathbf{d}_y)$.
As $X$ is a bounded subset of euclidean space and $\mathbf{d}$ is bounded from above by $1$, each $\gr(\mathbf{d}_x)$ is bounded and so $d_H(x,y) < \infty$ for every $x,y \in X$.
For any bounded definable functions $f,g: X \rightarrow R$ we have:
$$ d_H( \gr(f),\gr(g)) \leqslant \| f- g \|_\infty. $$
As $\mathbf{d}(x,y) = \| \mathbf{d}_x - \mathbf{d}_y\|_\infty$ the map $\id:(X,d)\rightarrow(X,d_H)$ is distance decreasing.
Furthermore, as $\| \mathbf{d}_x - \mathbf{d}_y\|_\infty \leqslant 1$ for any $x,y \in X$, $(X,d_H)$ is bounded.
We let $(B,d_H)$ be the metric space associated to the pseudometric $(X,d_H)$ and let $(\widetilde{B},d_H)$ be the definable completion of $(B,d_H)$.
As $(\widetilde{B}, \widetilde{\mathbf{d}}_H)$ is bounded, Lemma~\ref{lem:properhau} implies that $(\widetilde{B},d_H)$ is definably compact.
We define $\iota_X$ to be the natural distance decreasing map 
$$ \iota_X: (X,d) \rightarrow  (\widetilde{B},d_H). $$
\end{dfn}

\begin{lem}\label{lem:hausdorfembedding}
If $(X,d)$ is blue and $X$ is a bounded of subset of euclidean space then $\iota$ is injective.
Equivalently, if $(X,d)$ is blue and $X$ is a bounded subset of euclidean space then $(X,d_H)$ is a metric space.
\end{lem}

\begin{proof}
We suppose that $d_H$ is not a metric.
Let $x,y$ be distinct elements of $X$ for which  $d_H(x,y)=0$.
So 
$$\cl_e[\gr(\mathbf{d}_x)]=\cl_e[\gr(\mathbf{d}_{y})]$$
 and in particular $(x,0)\in \cl_e[\gr(\mathbf{d}_{y})]$.
This contradicts Lemma~\ref{lem:defofgoodpoint}.
\end{proof}

By Proposition~\ref{prp:keytechnical} we have:

\begin{cor}\label{cor:contintocompact}
Any definable metric space admits a continuous injection into a definably compact metric space.
\end{cor}

\begin{lem}\label{lem:dimineq}
Let $A\subseteq X$ be definable.
Then $$\dim[\partial_d(A)]<\dim(A).$$
If $\dim(A)=\dim(X)$ then $\dim[\inter_d(A)] = \dim(A)$.
\end{lem}

\begin{proof}
It suffices to prove the lemma for any definable metric space definably isometric to $(X,d)$.
We therefore assume that $(X,d)$ is blue.
Thus $d_{H}$ is a metric on $X$.
The dimension inequality for Hausdorff limits (item (3) of Theorem 3.1 in \cite{vdDLimit}) implies $\dim[\partial_{d_{H}}(A)]<\dim(A)$.
Continuity of $\id:(X,d)\rightarrow (X,d_{H})$ implies that the $d$-frontier of $A$ is contained in the $d_{H}$-frontier of $A$.
So $\dim[\partial_d(A)]<\dim(A)$.
Suppose that $\dim(A)=\dim(X)$.
As $A \setminus \inter_d(A) = \partial_d(X \setminus A)$, we have
$$\dim[ A \setminus \inter_{d}(A) ] < \dim(X\setminus A) \leqslant \dim(A)$$
and so $\dim[ \inter_d(A) ] = \dim(A)$.
\end{proof}
We use this dimension inequality to prove the following key technical lemma:
\begin{lem}\label{lem:keyy}
Let $d'$ be another definable metric on $X$.
Suppose that $\dim_{x}(X,d')=\dim(X)$ at almost every $x\in X$.
Then $\id:(X,d)\rightarrow (X,d')$ is continuous almost everywhere.
\end{lem}

\begin{proof}
Let $D\subseteq X$ be the set of points at which $\id:(X,d)\rightarrow (X,d')$ is not continuous.
We suppose towards a contradiction that $\dim(D)=\dim(X)$.
Let $\pi_1:X^2\times R\rightarrow X$ be the projection onto the first coordinate.
 Let $Q\subseteq X^2\times R$ be the set of $(x,x',t)$ such that 
$x\in B_{d'}(x',t)$ and $x$ is not in the $d$-interior of $B_{d'}(x',t)$.
That is,
$$ \{ x \in X : (x,y,t) \in Q \}=\partial_{d} [B_{d'}(y,t)] \quad \text{for all } (y,t)\in X\times R^>.$$
By definition $D=\pi_1(Q)$.
It follows that 
$$\dim \{ x \in X : (x,y,t) \in Q \} <\dim X \quad \text{for all } (y,t) \in X \times R^>.$$
By the fiber lemma for o-minimal dimension, $\dim (Q)<2 \dim (X) +1$. Now we get a lower bound on $\dim Q$.
Let $x\in D$. For some $t>0$, $x$ is not in the $d$-interior of $B_{d'}(x,t)$.
If $(y,t')\in X\times R$ satisfies $x\in B_{d'}(y,t')\subseteq B_{d'}(x,t)$ then $x$ is not in the $d$-interior of $B_{d'}(x,t)$
and $(x,y,t')\in Q$.
Let $y\in B_{d'}\left(x,\frac{t}{3}\right)$ and $\frac{t}{3}<s<\frac{2t}{3}$.
Then $x\in B_{d'}(x,\frac{t}{3})$ and by the triangle inequality, $B_{d'}(y,s)\subseteq B_{d'}(x,t)$, so $(y,s)\in Q_x$.
We have shown:
$$ B_{d'}\left(x,\frac{t}{3}\right)\times \left(\frac{t}{3},\frac{2t}{3}\right)\subseteq Q_x. $$
So for each $x\in D$, $\dim(Q_x)\geqslant \dim_x(X,d') + 1$.
We assumed that $\dim(D)=\dim(X)$ so we have $\dim(Q_x)=\dim(X) + 1$ at almost every $x\in D$.
So $\dim(Q) = 2\dim(X) + 1$.
Contradiction.
\end{proof}

\begin{cor}\label{cor:continuityidentity}
We have the following:
\begin{enumerate}
\item $\id:(X,d)\rightarrow (X,e)$ is continuous almost everywhere.
\item There is a partition $\cX$ of $X$ into cells such that if $Y \in \cX$ then $\id:(Y,d)\rightarrow (Y,e)$ is continuous.
\item If $\dim_x(X,d)=\dim(X)$ at almost every $x\in X$ then both $\id :(X,e) \rightarrow (X,d)$ and $\id: (X,d) \rightarrow (X,e)$ are continuous almost everywhere on $X$.
\item If $\dim_x(X,d) = \dim(X)$ at almost every $x \in X$ then almost every point $p \in X$ has an $e$- and $d$-open neighborhood $V \subseteq X$ such that $\id: (V,d) \rightarrow (V,d)$ is a homeomorphism.
\item If $(X,d)$ is definably separable then there is a partition $\cX$ of $X$ into cells such if $Y \in \cX$ then $\id:(X,d)\rightarrow (Y,e)$ is a homeomorphism.
\item If $(X,d)$ is definably separable then $(X,d)$ is definably isometric to a definable metric space $(X',d')$ such that $\id:(X',e)\rightarrow (X',d')$ is continuous and $X'$ is a locally closed subset of euclidean space.
\end{enumerate}
\end{cor}

\begin{proof}
$(i)$ follows immediately from Lemma~\ref{lem:keyy}.
$(iii)$ follows from $(i)$ and Lemma~\ref{lem:keyy}.
We prove $(ii)$ by induction on $\dim(X)$.
If $\dim(X)=0$ we can take the $\cX$ to be the singleton subsets of $X$.
Suppose $\dim(X) > 0$.
Let $C$ be the set of points at which $\id: (X,d) \rightarrow (X,e)$ is continuous.
We can apply the inductive assumption to $(X\setminus C,d)$ and partition $X\setminus C$ into cells $X_1,\ldots,X_n$ such that $\id: (X_i,d) \to (X_i,e)$ is continuous for every $i$.
We then take some partition of $C$ into cells $X_{n+1},\ldots,X_m$.
Let $\cX = \{ X_1,\ldots, X_n, X_{n + 1}, \ldots, X_m\}$.
We now prove $(iv)$.
Let $V\subseteq X$ be the set of points at which both $\id: (X,e) \rightarrow (X,d)$ and $\id: (X,d) \rightarrow (X,e)$ are continuous. Then $\id:(V,d)\rightarrow (V,e)$ is a homeomorphism.
$(iii)$ implies $\dim(X\setminus V)<\dim(X)$.
So $\dim(V) = \dim(X)$, so Lemma~\ref{lem:dimineq} implies that the $d$-interior of $V$ is almost all of $V$.
We let $V'$ be the intersection of the $d$-interior of $V$ with the $e$-interior of $V$ in $X$.
Then $V'$ is almost all of $V$ and $\id : (V',d) \rightarrow (V',d)$ is a homeomorphism.
$(v)$ is proven in the same way as $(ii)$.
We prove $(vi)$ using $(v)$.
Let $\cX$ be the partition provided by $(v)$ and let $X'$ be the disjoint union of the elements of $\cX$. 
Let $d'$ be the natural pushforward of $d$ onto $X'$.
Then $\id:(X',e)\rightarrow (X',d)$ is continuous.
$X'$ is locally closed as it is a disjoint union of cells.
\end{proof}
The next lemma is immediate from $(iv)$ above and cell decomposition:
\begin{lem}\label{lem:openinsep}
Let $\dim(X)=l$.
Suppose that $\dim_x(X,d)=l$ at almost every $x\in X$. 
There is a definable injection $((0,1)^l,e)\rightarrow (X,d)$ which gives a homeomorphism between $(0,1)^l$ and a $d$-open subset of $X$.
\end{lem}
We now definably compactify definably separable metric spaces.

\begin{prp}\label{prp:sepintocompact}
Every definably separable metric space is definably homeomorphic to a definable subspace of a definably compact metric space.
\end{prp}

To prove Propostion~\ref{prp:sepintocompact} it suffices to show that a definably separable $(X,d)$ is definably isometric to a definable metric space for which the map $\iota$ is a homeomorphism onto its image; this shows that $(X,d)$ is definably homeomorphic to a definable subspace of a definably compact metric space.
This is established by the lemma below and Corollary~\ref{cor:continuityidentity}.

\begin{lem}\label{lem:secondhausembedd}
Suppose that $\id:(X,e)\rightarrow (X,d)$ is continuous.
Then the topologies given by $d$ and $d_{H}$ agree.
\end{lem}

\begin{proof}
Continuity of $\id: (X,e) \rightarrow (X,d)$ implies that $(X,d)$ is blue, thus it is enough to show that $\id:(X,d_{H})\rightarrow (X,d)$ is continuous.
Let $x \in X$ and $\delta>0$.
Let $\epsilon>0$ be such that for all $y \in Y$, if $e(x,y)<\epsilon$ then $d(x,y)<\delta$.
We assume that $\epsilon<\delta$.
Suppose that $d_{H}(x,y)<\epsilon$.
There is a point on $\gr(d_y)$ whose euclidean distance from $(x,0)$ is at most $\epsilon$.
Namely we have a $x'$ such that the euclidean distance between $(x,0)$ and $(x',d_y(x'))$ is at most $\epsilon$.
This implies that $e(x,x')<\epsilon$ and that $d(y,x')<\epsilon$.
So $d(x,x')<\delta$ and the triangle inequality gives $d(x,y)<\epsilon + \delta < 2\delta$.
So $\id:(X,d_{H})\rightarrow (X,d)$ is continuous.
\end{proof}
To prove Theorem~\ref{main} it now suffices to show that every definably compact metric space is definably homeomorphic to a definable set equipped with its euclidean metric.
This is done in the next subsection.

\subsection{Covering Definably Compact Metric Spaces}
\textit{In this subsection $(X,d)$ is a definably compact metric space.}
We want to show that:

\begin{clm}\label{clm:cvrclm}
There are definably compact sets $Z_1,\ldots,Z_n$ and definable continuous maps $\rho_i:(Z_i,e)\rightarrow (X,d)$ whose images cover $X$.
\end{clm}

Suppose that we have such $Z_i$ and $\rho_i$.
Let $Z$ be the disjoint union of the $Z_i$ and let $\rho:(Z,e)\rightarrow (X,d)$ be the map induced by the $\rho_i$.
Then $(Z,e)$ is definably compact and $\rho$ is continuous, Lemma~\ref{lem:coveringdefcompacts} gives a definable set $Y$ and a definable homeomorphism $(X,d)\rightarrow (Y,e)$.
Our construction of the $Z_i$ is an application of the definable Michael's Selection Theorem, Theorem 4.1 of \cite{Aschenbrenner-Thamrongthanyalak}, which we now recall.
Let $T\subseteq E\times R^m$. 
We say that $T$ is \textbf{lower semi-continuous} if for every $(x,y)\in T$ and neighborhood $V\subseteq R^m$ of $y$ there is a neighborhood $U$ of $x$ such that $T_{y}\cap V\neq\emptyset$ for every $y \in U$.

\begin{thm}[Aschenbrenner-Thamrongtanyalak]
 Let $E$ be a definable locally closed set and let $T\subseteq E\times R^m$ be definable and lower semi-continuous such that each $T_x$ is closed and convex. 
 Then there is a definable continuous function $\sigma:E\rightarrow R^m$ such that $\sigma(x)\in T_x$ always.
\end{thm}
We will use the following lemma to apply Michael Selection:

\begin{lem}
Let $E \subseteq R^l$ and let $S\subseteq E\times R^m$ be lower semi-continuous.
Let $T\subseteq E\times R^m$ be such that for every $x \in E$, $T_x$ is the closure of the convex hull of $S_x$ in $R^m$.
Then $T$ is lower semi-continuous.
\end{lem}

\begin{proof}
Fix $x\in E$ and let $V$ be an open subset of $R^m$ such that $T_x \cap V\neq\emptyset$.
The intersection of $V$ with the convex hull of $S_x$ is nonempty, hence there are
$p,p'\in S_x$  and $s\in(0,1)$ such that $sp+(1-s)p'\in V$.
There are open $W,W'\subseteq R^m$ such that $p\in W$, $p'\in W'$ and if $(q,q')\in W\times W'$
then $sq+(1-s)q'\in V$. 
Let $U$ be a neighborhood of $x$ such that if $y \in U$ then $S_{y}\cap W$ and $S_{y}\cap W'$ are both nonempty.
Fix $q\in S_{y}\cap W$ and $q'\in S_{y}\cap W'$.
Then $sq+(1-s)q'$ is an element of $T_{y}\cap V$. 
So $T$ is lower semi-continuous.
\end{proof}
We now prove Claim~\ref{clm:cvrclm}.
\begin{proof}
As it suffices to prove the claim for any definable metric space definably isometric to $(X,d)$ we assume that $X$ is a bounded subset of euclidean space.
We apply induction on $\dim(X)$.
The base case $\dim(X)=0$ is trivial, we treat the case $\dim(X)>0$.
Applying Proposition~\ref{cor:continuityidentity} we partition $X$ into cells $X_1,\ldots,X_n$ such that each $\id:(X_i,d)\rightarrow (X_i,e)$ is a homeomorphism. 
For each $i$ we construct a definably compact definable set $Z_i$ and a definable continuous surjection 
$$\rho_i: (Z_i,e) \rightarrow (\cl_d(X_i),d). $$
Fix $i$.
To simplify notation we let $C=X_i$ and $A=\partial_d(C)$. 
As $(C,e)$ is locally definably compact so $(C,d)$ is locally definably compact.
By Lemma~\ref{lem:closureincomplete} this implies that $C$ is a locally closed subset of $X$.
Then $C$ is open in $\cl_d(C)$, so $A$ is a $d$-closed subset of $X$.
Thus $(A,d)$ is definably compact.
By Lemma~\ref{lem:dimineq} we have $\dim(A)<\dim(X)$, so we apply the inductive assumption to $(A,d)$.
We let $A'\subseteq R^m$ be a definably compact subset of euclidean space and let $\tau:(A,d)\rightarrow (A',e)$ be a homeomorphism.
We let $Y$ be the disjoint union of $X_i$ and $A'$.
We define $\tau':\cl_d(X_i)\rightarrow Y$ to be identity on $X_i$ and $\tau$ on $A$.
We let $d'$ be the pushforward of $d$ by $\tau$.
Then $(Y,d')$ is definably isometric to $(\cl_d(X_i),d)$, so it suffices to construct a definable continuous surjection $\rho:(Z,e)\rightarrow (Y,d')$ from a definably compact $Z$.
We assume without loss of generality that the $d$- and $e$-topologies agree on $A$.
Then $A$ is a definably compact subset of euclidean space. 
We use the definable Michael's Selection Theorem to construct a bounded continuous definable function $\sigma:C \rightarrow R^m$
 which satisfies the following for any path $\gamma$ in $C$: if $\gamma$ $d$-converges to $x\in A$ then $\lim_{t\rightarrow 0^+}(\sigma\circ \gamma)(t)=x$.

We first suppose we have such a $\sigma$.
Let $Z$ be the closure of $\gr(\sigma)$ in the ambient euclidean space.
Boundedness of $C$ and $\sigma$ ensures that $Z$ is definably compact.
Let $\pi_1:Z\rightarrow C$ and $\pi_2:Z\rightarrow R^m$ be the coordinate projections.
  We define a map $\rho:Z\rightarrow \cl_d(C)$ by setting $\rho(x)=\pi_1(x)$ when $x\in \gr(\sigma)$ and $\rho(x)=\pi_2(x)$ when $x\in\partial_e[\gr(\sigma)]$. 
First we must show that $\rho$ does in fact take values in $\cl_d(C)$.
It suffices to show that $\pi_2(x)\in A$ when $x\in\partial_e[\gr(\sigma)]$, to this effect fix such an $x$.
Let $\gamma$ be a path in $\gr(\sigma)$ which $e$-converges to $x$.
Then $\pi_1\circ \gamma$ is a path in $C$ which by definable compactness must $d$-converge to some $y\in \cl_d(C)$. 
This $y$ must be an element of $A$, otherwise $\pi_1\circ \gamma$ would $e$-converge to an element of $C$ and the continuity of $\sigma$ would force $x\in\gr(\sigma)$.
We have
$$\rho(x)=\pi_2(x)=\lim_{t\rightarrow 0^+}(\pi_2\circ \gamma)(t)=\lim_{t\rightarrow 0^+}(\sigma\circ\pi_1\circ \gamma)(t)=y\in\cl_d(C).$$
We have shown that $\rho$ takes values in $\cl_d(C)$. 

We now show that $\rho$ is surjective.
As $C$ is obviously contained in the image of $\rho$ it suffices to show that $A$ is contained in the image of $\rho$.
Let $y\in A$. 
Applying definable choice let $\gamma$ be a path in $C$ which $d$-converges to $y$. 
We have $\lim_{t\rightarrow 0^+}(\sigma\circ \gamma)(t)=y$.
Letting $z$ be the $e$-limit of $(f(t),(\sigma \circ \gamma)(t))$ as $t\rightarrow 0^+$ we have $z\in Z$ and $\rho(z)=y$.

We now show that $\rho$ is continuous.
We first show that $\rho$ is continuous at every point in $\gr(\sigma)$.
The restriction of $\rho$ to $\gr(\sigma)$ is a continuous function, as projections are continuous and $\id:(C,e)\rightarrow (C,d)$ is continuous. 
As $\gr(\sigma)$ is a cell it hence open in its closure, $Z$. 
Thus $\rho$ is continuous at every point in $\gr(\sigma)$.
We now show that $\rho$ is continuous at every point in $\partial_e[\gr(\sigma)]$. 
Fix $x\in\partial_e[\gr(\sigma)]$.
Let $\gamma$ be a path in $Z$ which $e$-converges to $x$.
 First suppose that $\gamma(t)\in[\gr(\sigma)]$ when $t$ is sufficiently small. 
 Then the $d$-limit of
$(\rho\circ \gamma)(t)$ as $t\rightarrow 0^+$ equals the $d$-limit of $(\pi_1\circ \gamma)(t)$ as $t\rightarrow 0^+$, which equals
$\lim_{t\rightarrow 0^+}(\sigma\circ \pi_1\circ \gamma)(t)$. 
As we showed above this equals $\rho(x)$.
We now assume $\gamma(t)\in\partial_e[\gr(\sigma)]$ when $0<t\ll 1$.
As $\pi_2:(\partial_e[\gr(\sigma)],e)\rightarrow (A,e)$ and $\id:(A,e)\rightarrow (A,d)$ are continuous, the restriction of $\rho$ to $\partial_e[\gr(\sigma)]$ is continuous.
This implies that $\gamma(t) \rightarrow x$ as $t \rightarrow 0^+$.

We finally construct $\sigma$.
Let $B\subseteq C\times R^m$ be the set of $(x,y)$ such that $y\in A$ and $d(y,x)<2d(A,x)$. 
As $A$ is $d$-closed, $d(A,x)$ is always strictly positive, so $B_x$ is always nonempty.
We show that $B$ is lower semi-continuous.
Fix $z,y\in C$ satisfying $y \in B_z$.
 By the definition of $B$, $d(y,z)<2d(A,z)$.
We show that if $z'\in C$ satisfies
 $$d(z,z')<\frac{1}{4} | 2d(A,z) - d(y,z) |$$
then $y\in B_{z'}$. 
As the set of such $z'$ is an $e$-open subset of $C$ this gives lower semi-continuity.
We have
$$ | d(y,z') - d(y,z) | \leqslant d(z,z') <\frac{1}{4} | 2d(A,z) - d(y,z) | $$
 and 
$$|2d(A,z)-2d(A,z')|\leqslant 2d(z,z') < \frac{1}{2} | 2d(A,z) -d(y,z) |. $$
These two inequalities give $d(y,z') < 2d(A,z')$ so $y\in B_{z'}$.

 Let $D\subseteq X\times R^m$ be the definable set such that for each $x\in C$, $D_x$ is the closure of the convex hull of $B_x$. 
Each $D_x$ is closed and convex and $D$ is lower semi-continuous.
Applying the Michael's Selection Theorem let $\sigma:C\rightarrow R^m$ be a continuous definable map
 such that $\sigma(x)\in D_x$ holds for all $x \in C$.
Fix a path $\gamma$ in $C$ which $d$-converges to $x\in A$. 
We show that $\lim_{t\rightarrow 0^+}(\sigma\circ \gamma)(t)=x$.
Towards this we show that $B_{\gamma(t)}$ is contained in the $d$-ball with center $x$ and radius $3d(x,\gamma(t))$.
Fix $w \in B_{\gamma(t)}$.
By definition of $B$ we have $d(w,\gamma(t))<2d(A,\gamma(t))$. 
As $x\in A$ we have
$d(A,\gamma(t))\leqslant d(x,\gamma(t))$
and we compute:
$$d(x,w)\leqslant d(\gamma(t),x)+d(\gamma(t),w) \leqslant d(\gamma(t),x) + 2d(A,\gamma(t)) \leqslant 3d(\gamma(t),x).  $$
As $(A,d)$ is definably compact and the $d$-topology agrees with the $e$-topology on $A$, we apply Lemma~\ref{prp:compactuniformcontinuity} to produce a definable $h:R^{\geqslant}\rightarrow R^{\geqslant}$ satisfying 
$$\| y - y'\| \leqslant h(d(y,y')) \quad \text{for all } y,y'\in A$$
 and such that $\lim_{t\rightarrow 0^+}h(t) =0$.
We let $r(t)=h(3d(\gamma(t),x))$.
Then $r(t)\rightarrow 0$ as $t\rightarrow 0^+$ and  $B_{\gamma(t)} \subseteq B_e(x,r(t))$ for all $t \in R^>$.
As euclidean balls are convex, 
$$C_{\gamma(t)} \subseteq B_e(x, r(t)) \quad \text{for all } t \in R^>.$$
Thus $\| (\sigma\circ \gamma)(t) - x \| \leqslant r(t)$. 
The path $\sigma \circ \gamma$ $e$-converges to $x$.
\end{proof}
This completes the proof of Theorem~\ref{main}.
The uniform version of Theorem~\ref{main} given below follows by applying model-theoretic compactness and the existence of definable Skolem functions in the usual way:

\begin{cor}\label{cor:defsep}
Let $\{(X_\alpha,d_\alpha): \alpha\in R^l\}$ be a definable family of metric spaces.
The set of $\alpha \in R^l$ such that $(X_\alpha,d_\alpha)$ is definably separable is definable.
In fact, there is a definable family $\{ A_\alpha : \alpha \in R^l\}$ of sets such that $A_\alpha \subseteq X_\alpha$, a definable family of sets $\{ Z_\alpha : \alpha \in R^l\}$ and a definable family of functions $h_\alpha : X_\alpha \rightarrow Z_\alpha$ such that for every $\alpha \in R^l$ exactly one of the following holds:
\begin{enumerate}
\item $A_\alpha$ is a $d$-discrete subset of $X_\alpha$,
\item $h_\alpha$ gives a homeomorphism $(X_\alpha, d_\alpha) \rightarrow (Z_\alpha, e)$.
\end{enumerate}
The set of $\alpha$ for which $h_\alpha$ gives such a homeomorphism is definable.
\end{cor}
Recall our standing assumption that $(X,d)$ is definably compact.
Combining Corollary~\ref{cor:contintocompact} and Theorem~\ref{main} we have:
\begin{cor}\label{cor:contodef}
Every definable metric space $(Y,d)$ is definably isometric to a definable metric space $(Z,d')$ such that $\id:(Z,d')\rightarrow (Z,e)$ is continuous.
\end{cor}

\begin{proof}
After replacing $(Y,d)$ with a definably isometric space if necessary we can suppose that $d_{H}$ is a metric on $Y$ and that $\id:(Y,d)\rightarrow (Y,d_{H})$ is continuous.
There is a definable set $Z$ and a definable homeomorphism
$$ \tau: (Y,d_{H}) \rightarrow (Z,e). $$
Let $d'$ be the pushforward of $d$ onto $Z$ by $\tau$.
Then $(Z,d')$ is definably isometric to $(Y,d)$.
We show that $\id: (Z,d') \rightarrow (Z,e)$ is continuous by factoring it as a composition of continuous maps:
$$ (Z,d') \overset{\tau^{-1}}{\longrightarrow}   (Y,d) \overset{\id}{\longrightarrow} (Y,d_{H}) \overset{\tau}{\longrightarrow} (Z,e).$$
\end{proof}

\section{Product Structure of General Definable Metric Spaces}
In this section we prove Proposition~\ref{prp:productstructure}.
As an application in the next section we prove Corollary~\ref{cor:lasttopdim}, which characterizes the topological dimension of a metric space definable in an o-minimal expansion of the real field.

\begin{prp}\label{prp:productstructure}
Almost every $x\in X$ is contained in a definable $e$- and $d$-open set $U\subseteq X$ which admits a definable homeomorphism
$$ (U,d) \rightarrow (I^n,e) \times (D,d_{\di}) $$
where $I$ is an open interval, $n \leqslant \dim(X)$, and $D$ is a definable set.
If $n = \dim(X)$ then we can take $D$ to be a singleton.
\end{prp}

\begin{proof}
Let $\dim(X) = m$.
It suffices to prove the proposition for any definable metric space definably isometric to $(X,d)$ as any definable isometry $(X,d) \rightarrow (X',d')$ locally gives a homeomorphism $(X,e) \rightarrow (X',e)$ at almost every point.
Therefore after applying Corollary~\ref{cor:contodef} we assume that $\id:(X,d)\rightarrow (X,e)$ is continuous.
Then every $e$-open subset of $X$ is $d$-open.
It suffices to fix an $e$-open, $m$-dimensional definable $W\subseteq X$ and find an $e$-open definable $U\subseteq W$ which satisfies the conditions of the proposition. 
We suppose, after shrinking $W$ if necessary, that $\dim_x(X,d)=n$ for every $x\in W$. 
Then $\dim_x(W,d)=n$ at every $x\in W$.
If $n=m$ we use Lemma~\ref{lem:openinsep} to find an $e$-open definable $U\subseteq W$ for which $\id:(U,d) \rightarrow (U,e)$ is a homeomorphism and the proposition holds with $D$ a singleton.
We therefore assume that $n<m$.
Let $f: W \rightarrow R^>$ be a definable function such that $\dim B(p, f(p)) = n$ for all $p \in W$.
For $x \in W$ let $B_x = B(x,f(x))$.
Applying a version of the good directions lemma in the same way as in the proof Lemma~\ref{lem:firstsep} we let $\pi:W \rightarrow R^n$ be a linear projection and $W' \subseteq W$ be an $e$-open $m$-dimensional definable set such that
$$ | \pi^{-1}(w) \cap B_x| < \infty \quad \text{ for all } w \in \R^n,x \in W. $$
After replacing $W$ with $W'$ if necessary we suppose that these properties hold for $W$.
As in the proof of Lemma~\ref{lem:firstsep} each fiber of $\pi$ is $d$-discrete.
For all $x \in W$ there is a $t \in R^>$ such that $d(x,y) > t$ whenever $y \in W$ is distinct from $x$ and satisfies $\pi(x) = \pi(y)$.
After applying Lemma~\ref{lem:c} and replacing $W$ with a smaller $m$-dimensional open set if necessary we fix a $t>0$ such that if $x,y\in W$ and $\pi(x) = \pi(y)$ then $d(x,y) > t$.
After replacing $f$ with the function $\max\{ f, \frac{1}{2}{t}\}$ if necessary we suppose that each $B_x$ has radius at most $\frac{1}{2}t$.
Note that we still have $\dim(B_x) = n$ and $\dim_y(B_x)=n$ at all $y\in B_x$.
As $\diam_d(B_x)<t$ the restriction of $\pi$ to each $B_x$ is injective.
Furthermore if $B_x \cap B_{y} \neq \emptyset$ then the triangle inquality implies $d(x,y) < t$.
Thus if $\pi(x) = \pi(y)$ then $B_x$ and $B_{y}$ have empty intersection.
Injectivity of $\pi$ on $B_x$ implies $\dim \pi(B_x) = n$.
As $\dim \pi(W) = n$ the fiber lemma for o-minimal dimension gives that some fiber of $\pi$ is $(m-n)$-dimensional.
We fix an $(m-n)$-dimensional fiber $F\subseteq W$ and let $G=\bigcup_{x\in F} B_x$.
As the union is disjoint $\dim(G)=m$.
As $\dim_y(B_x,d) = n$ at every $y\in B_x$ we apply Lemma~\ref{lem:openinsep} uniformly we take $\beta: F \times (0,1)^n \rightarrow G$ to be a definable function such that for each $x\in F$, $\beta_x[(0,1)^n]$ is an $e$-open subset of $B_x$ and 
$$\beta_x: ((0,1)^n,e)\rightarrow (\beta_x(I^n),d)$$
 is a homeomorphism for each $x \in W$.
We fix a definable $D\subseteq F$ with $\dim(D)=m-n$ and an interval $I\subseteq (0,1)$ such that the restriction of $\beta$ to $D \times I^n$ gives a continuous map between euclidean topologies.
We now replace $\beta$ with the restriction of $\beta$ to $D \times I^n$.
By the o-minimal open mapping theorem, see \cite{woe} or \cite{jj}, $\beta(D \times I^n)$ is an $e$-open subset of $W$.
We let $U=\beta(D \times I^n)$. 
$U$ is a $d$-open subset of $X$.
Then $(U,d)$ is, as a topological space, the disjoint union of the sets $\beta(\{b\}\times I^n)$.
It follows that
$$ \beta: (D,d_{\di}) \times (I^n,e) \rightarrow (U,d) $$
is a homeomorphism.
\end{proof}
The following two corollaries are weakenings of the previous proposition.
We see in particular that every definable metric space is ``almost locally definably compact".

\begin{thm}\label{cor:localhomeo}
Let $(X,d)$ be an arbitrary definable metric space.
Almost every $x\in X$ has a $d$-neighborhood $U$ such that $id:(U,d)\rightarrow (U,e)$ is a homeomorphism.
\end{thm}

\begin{thm}\label{cor:localhomeo2}
Almost every $x\in X$ has a $d$-neighborhood which is definably homeomorphic to an open subset of some $R^k$.
\end{thm}
Note that we could have $k=0$.

\section{Topological Dimension of Definable Metric Spaces}\label{section:top}
\textit{In this section we assume that $\cR$ expands the ordered field of real numbers.}
We describe the topological dimension of a definable metric space $(X,d)$.
We first recall the definition and some facts about topological dimension.
The notion of topological dimension that we use is frequently called the ``small inductive dimension''.

\subsection{Topological Dimension}
In this subsection $X$ is a topological space.
We inductively define the topological dimension, $\dim_{\typ}(X)$, of $X$:
\begin{enumerate}
\item $\dim_{\typ}(X) = -1$ if and only if $X = \emptyset$.
\item $\dim_{\typ}(X) \leqslant n$ if and only if $X$ has a neighborhood basis consisting of sets whose boundaries have topological dimension at most $n - 1$,
\end{enumerate}
and $\dim_{\typ}(X) \in \mathbb{N}\cup\{\infty\}$ is the least $n$ such that $\dim_{\typ}(X) \leqslant n$.
We will need four facts about the topological dimension.
The first can be proven with an easy inductive argument:

\begin{fact}\label{fact:monotoncetopdim}
If $A \subseteq X$ then $\dim_{\typ}(A) \leqslant \dim_{\typ}(X)$.
\end{fact}
The second is essentially trivial:

\begin{fact}\label{fact:disjointuniontopdim}
Let $\{ X_i : i \in I \}$ be a family of topological space and suppose $X$ is the disjoint union of the $X_i$.
Then
$$ \dim_{\typ}(X) = \sup \{ \dim_{\typ}(X_i) : i \in I \}.$$
\end{fact}
The last two require more effort, see \cite{Hurewicz-Wallman} for a full account.

\begin{fact}
The topological dimension of $\bR^n$ equipped with the euclidean topology is $n$.
\end{fact}

\begin{fact}\label{fact:closeduniontopdim}
Let $\{ F_i  : i \in \mathbb{N} \}$ be a family of closed sets which cover $X$.
Then 
$$ \dim_{\typ}(X) = \sup \{ \dim_{\typ}(F_i) : i \in \mathbb{N} \}. $$
\end{fact}

\subsection{Topological Dimension of Definable Metric Spaces}
In this subsection $(X,d)$ is a definable metric space.
\begin{thm}\label{cor:lasttopdim}
The topological dimension of $(X,d)$ is the maximal $m$ for which there is a definable continuous injection $(I^m,e) \rightarrow (X,d)$ for an open interval $I \subseteq \mathbb{R}$.
\end{thm}

\begin{proof}
Suppose $g: (I^m, e) \rightarrow (X,d)$ is a definable continuous injection.
Then
$$ \dim_{\typ}(X,d) \geqslant \dim_{\typ}(g(I^m),d) = \dim_{\typ}(I^m,e) = m. $$
We now suppose that there is no definable continuous injection $(I^m,e) \rightarrow (X,d)$ and show that $\dim_{\typ}(X,d) < m$.
We apply induction to $\dim(X)$.
Therefore we use Proposition~\ref{prp:productstructure} to partition $(X,d)$ into definable sets $U,A$ such that:
\begin{enumerate}
\item $U$ is $e$-open and $\dim(A) < \dim(X)$.
\item Every point in $U$ has an $e$-neighborhood $V$ which admits a definable homeomorphism
$$ (V,d) \rightarrow (C,e) \times (D,d_{\di}) $$ 
for definable sets $C,D$ with $\dim(C) < m$.
\end{enumerate}
There is no definable continuous injection $(I^m,e) \rightarrow (A,d)$, so the inductive assumption implies that $\dim_{\typ}(A,d) < m$.
As $(U,e)$ is separable we can cover $U$ with countably many $e$-closed definable sets $\{ F_i : i \in \mathbb{N} \}$ such that each $F_i$ is contained in an $e$-open $V\subseteq X$ which satisfies the conditions on $V$ in the statement.
If $V$ satisfies the conditions in the statement by Fact~\ref{fact:disjointuniontopdim} then $\dim_{\typ}(V,d) = \dim(C) < m$, so $\dim_{\typ}(F_i,d) < m$.
As
$$ X = A \cup \bigcup_{ i \in \mathbb{N} } F_i, $$
Fact~\ref{fact:closeduniontopdim} gives $\dim_{\typ}(X,d) \leqslant m$.
\end{proof}
As an immediate consequence we have:
\begin{cor}\label{cor:topdimhaussanalaogue}
$\dim_{\typ}(X,d) \leqslant \dim(X)$.
\end{cor}

\bibliographystyle{uclathes}
\bibliography{bilio}

\end{document}